\documentclass[12pt]{article}

\usepackage{amsmath}
\usepackage{amsthm}
\usepackage{amsfonts}
\usepackage{epsfig}
\newtheorem{proposition}{Proposition}[section]
\newtheorem{definition}[proposition]{Definition}
\newtheorem{lemma}[proposition]{Lemma}
\newtheorem{theorem}[proposition]{Theorem}
\newtheorem{corollary}[proposition]{Corollary}

\newdimen\AAdi%
\newbox\AAbo%
\def\AArm{\fam0 }
\def\AAk#1#2{\setbox\AAbo=\hbox{#2}\AAdi=\wd\AAbo\kern#1\AAdi{}}%
\def\AAr#1#2#3{\setbox\AAbo=\hbox{#2}\AAdi=\ht\AAbo\raise#1\AAdi\hbox{#3}}%
\def\BBone{{\AArm 1\AAk{-.8}{I}I}}%

\newcommand {\CB}{{\cal B}}
\newcommand {\CC}{{\cal C}}

\newcommand {\CF}{{\cal F}}
\newcommand {\CG}{{\cal G}}

\newcommand {\CK}{{\cal K}}
\newcommand {\CL}{{\cal L}}
\newcommand {\CM}{{\cal M}}

\newcommand {\CO}{{\cal O}}
\newcommand {\CP}{{\cal P}}

\newcommand {\CV}{{\cal V}}

\newcommand{\disp}{\displaystyle}
\newcommand{\eps}{\varepsilon}
\newcommand{\8}{\infty}

\def\lu{\lambda^u}  
\def\ls{\lambda^s}

\def\m1{{-1}}

\newcommand{\ninf}{{n\rightarrow+\8}}
\newcommand{\ol}{\overline}

\def\s{\sigma}

\newcommand{\wh}{\widehat}
\newcommand{\wt}{\widetilde}

\newcommand{\inte}[1]{\stackrel{\circ}{#1}}

\newcommand{\N}{\Bbb{N}}
\newcommand{\R}{\Bbb{R}}

\newcommand{\Z}{\Bbb{Z}}

\def\picture #1 by #2 (#3){  
\vbox to #2{
\hrule width #1 height 0pt depth 0pt
\vfill
\special{picture #3} 
}}
\def\scaledpicture #1 by #2 (#3 scaled #4){{
\dimen0=#1 \dimen1=#2
\divide\dimen0 by 1000 \multiply\dimen0 by #4
\divide\dimen1 by 1000 \multiply\dimen1 by #4
\picture \dimen0 by \dimen1 (#3 scaled #4)}}

\theoremstyle{definition}
\newtheorem{remark}{Remark}

\usepackage{pdfsync}
\usepackage{graphicx}
\usepackage{epstopdf}
\usepackage{hyperref}
\usepackage{float}
\usepackage{amssymb}
\usepackage{subeqnarray}
\usepackage{color}

\def \wh {\widehat }

\def\K{\mathbb{K}}

\newcommand{\amsc}[1]{\medskip\noindent{\bf\small{AMSC}:\ }{#1}\\}

\newcommand{\keywords}[1]{\medskip\noindent{\bf \small{Keywords}:\ }{#1}\\}

\newenvironment{proofof}[1]{\medskip \noindent{\bf Proof of #1.}}{ \hfill\qed\\ }

\newcommand{\diam}{\mbox{diam\,}}

\renewcommand{\P}{\mathbb{P}}
\def\deriv#1#2{\frac{\partial#1}{\partial#2}}
\newcommand{\ie}{{\it i.e.} }
\newcommand{\eg}{{\it e.g.} }
\def\be{\beta}
\def\al{\alpha}
\begin{document}
\synctex=1

\title{ Thermodynamic formalism for Lorenz maps}
\author{Renaud Leplaideur \& Vilton Pinheiro}

\thispagestyle{empty}

\maketitle

\begin{abstract}
For a 2-dimensional map representing an expanding geometric Lorenz attractor we prove that the attractor is the closure of a union of as long as possible unstable leaves with ending points. This allows to define  the notion of good measures, those giving full measure to the union of these open leaves. 

Then, for any H\"older continuous potential we prove that there exists at most one relative equilibrium state among the set of good measures. 
Condition yielding existence are given. 
  \end{abstract}

\amsc{37A30; 37A60; 37D25; 37D35; 37D45; 37E.}

\section{Introduction and statement of results}
\subsection{Background}

The Lorenz attractor was introduced in \cite{lorenz} to develop a simplified mathematical model for atmospheric convection. It turns out that this system of differential equations became a famous example of chaotic dynamical system. 
Perhaps one of the most spectacular fact, at least from the mathematical point of view, is that the true proof of the existence of the attractor is computer assisted (see \cite{tucker}) and, as far as we know, there is no other proof.

The system of differential equations actually generates a partially hyperbolic flow in $\R^{3}$. It has been extended (see \eg \cite{rovella}) to Lorenz-like attractors. There exists a large literature concerning their topological dynamical properties (see \eg \cite{arroyo-pujals, guckenheimer, guckenheimer-williams, hubbard-sparrow,luzzatto-pacaud, milnor, williams}).

\bigskip
The present paper focus on the ergodic properties, more precisely on the thermodynamic formalism. For this topic, the literature is less abundant, and usually focus on physical measures and/or more generally on conformal measures with respect to the $uc$-Jacobian (see \eg \cite{keller-stpierre,pacifico-todd}). Here we work with general H\"older continuous potentials. 

It is now a classical way to consider induction for studying the thermodynamic formalism for non-uniformly hyperbolic dynamical systems. Several relatively different methods exist (see \eg \cite{leplaideur1, pesin-senti, bruin-todd, pinheiro}). But for all them, the main problem is that they may exist measures which are not seen by the induction; therefore, it is always needed to check if the maximum obtained is the global one. In other words, liftable measures are ``good'' measures and it is hoped that they will be relevant with respect to the thermodynamic formalism.

Here also, we define a notion of ``good'' measures, and they turn out to be the ones which can be lifted by our induction scheme. However, and we guess this is a novelty, a good (ergodic) measure is not defined by this property but by the locus of its ergodic component. Actually, we describe the attractor and  identify a``good'' locus which much have full measure for any ``good'' measure (see Definition \ref{def-measuregood}). The ``good'' measures carry the relevant part of the dynamics\footnote{ in fact all the dynamics except the singularities}, and we prove here the uniqueness of the equilibrium state among the set of good measures (see Theorem D).

\bigskip
Our description of the ``good'' locus is a consequence of a sensibly different approach for the study of the dynamics. 
Indeed, a classical way to study Lorenz-like  attractors is to consider the return map into a transversal Poincar\'e section (see Fig. \ref{Fig-the map2dim}), and then, the associated one-dimension dynamics. 
In that approach, the existence of the unstable  direction and then the unstable leaves is a direct consequence of the global definition as a partially hyperbolic flow with an invariant splitting 
$$\R^{3}=E^{u}\oplus E^{c}\oplus E^{s}.$$
Our approach is here different. We start with a dynamics on a two-dimensional square (which should play the role of the Poincar\'e section) and recover the existence of the unstable direction (see Theorem A) and the unstable leaves. For that we use classical tools as the graph transform theory. As a by-product, we get the description of the ``good'' part of the attractor (see Theorem C) and then the definition of ``good'' measures.


\subsection{Settings}
%

The map $F$ maps $[-1,1]^{2}$ into itself. It is  a fibered map of the form $F(x,y):=(f(x),g(x,y))$. 
The map $f$ has a discontinuity at 0 and is defined by two branches. Each branch is an increasing $\CC^{1+}$-diffeomorphism  respectively from $[-1,0[$ onto $[v_{l},1[\varsubsetneq]-1,1[$ and from $]0,1]$ onto $]-1,v_{r}]\varsubsetneq]-1,1[$ with  $-1<v_{l}<0<v_{r}<1$ (see Figure \ref{Fig-the map1dim}).
It is expanding and for simplicity we assume that 
$$f'\ge\sqrt 2$$
holds\footnote{The same work could be done with $f'>1+\eps>1$.}. 
At $0$, $f$ is thus bi-valued: $f(0^{-})=1$ and $f(0^{+})=-1$.
In addition of being a discontinuity point, $0$ is also a \emph{non-flat} critical point for the derivative: $f'(x)$ is of the form $\disp \frac1{|x|^{1-\rho}}$ for $\rho\in(0,1)$  and $x$ close to 0. Consequently, $\lim_{x\to 0^{\pm}}f'(x)=+\8$. 

\begin{figure}[htbp]
\begin{center}
\includegraphics[scale=0.3]{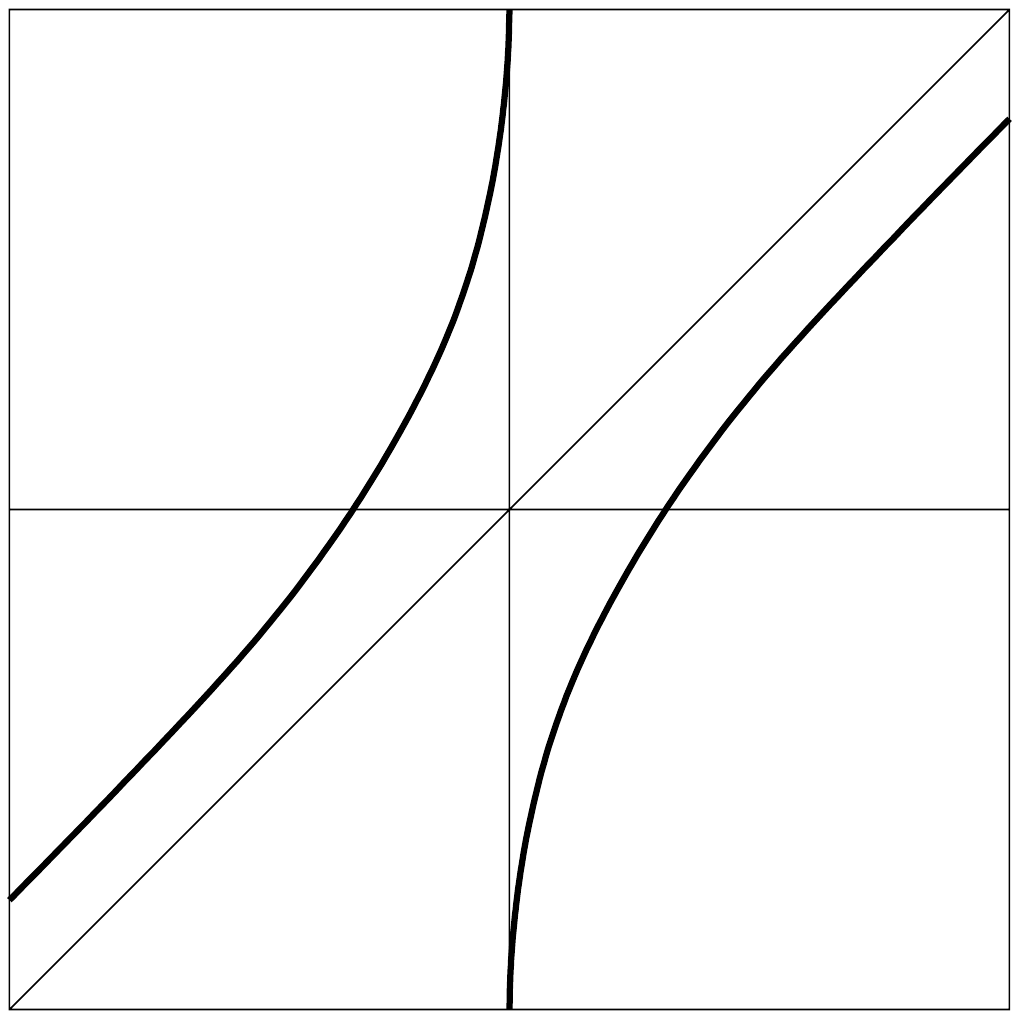}
\caption{1-dimensional map $f$}
\label{Fig-the map1dim}
\end{center}
\end{figure}

\medskip
The map $g$ is the continuation in $[-1,1]^{2}$ of a $\CC^{1+}$-map defined on $[-1,0)\cup(0,1]\times[-1,1]\to[-1,1]$. 
It satisfies $g(0^{-},y)=y_{+}>0$ and $g(0^{+},y)=y_{-}<0$, for every $y$, where $y_{\pm}$ are in $]-1,1[$. Moreover, $\disp\left|\frac{\partial g}{\partial y}\right|\le \frac12$ and there exists some positive real number $M$ such that $\disp\left|\frac{\partial g}{\partial x}\right|\le M$. 

Consequently, $F$ maps $[-1,1]^{2}$ strictly into itself, and maps each left and right part of the square onto two disjoint  ``triangles'' as on Figure \ref{Fig-the map2dim}. 
It is thus one-to-one, except on the vertical line $\{0\}\times[-1,1]$. 

\begin{figure}[htbp]
\begin{center}
\includegraphics[scale=0.4]{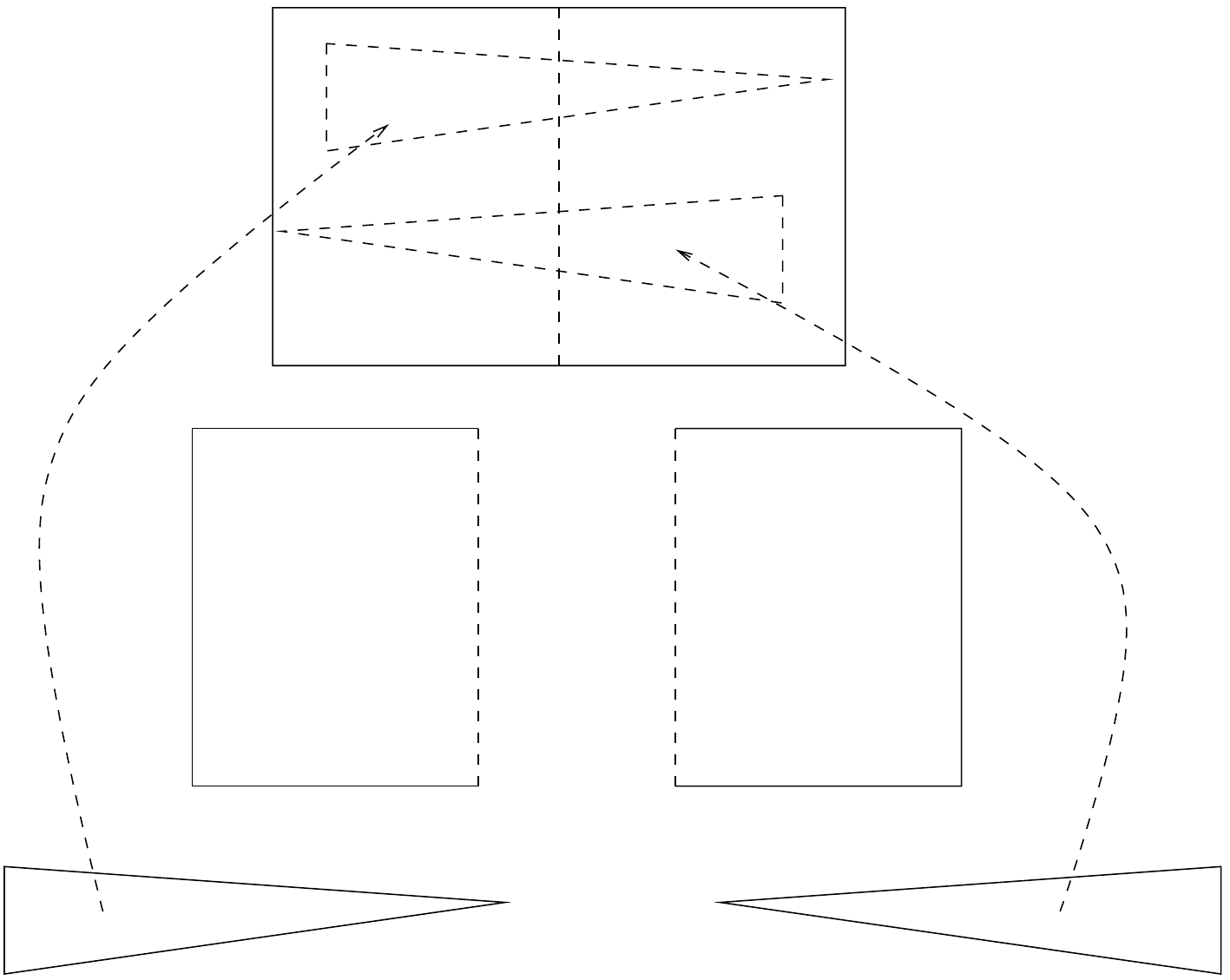}
\caption{2-dimensional map}
\label{Fig-the map2dim}
\end{center}
\end{figure}

The attractor is by definition the set $\Lambda:=\bigcap_{n\ge 0}F^{n}([-1,1]^{2})$. 
In $[-1,1]$ the point 0 is referred to as the {\em critical point} for the dynamics of $f$. The {\em post-critical orbits} are the respective forward orbits of $1$ and $-1$. They are denoted by $\CO_{f}^{+}(0)$.

\paragraph{Assumption.} {\bf In all the paper we assume that $0$ is neither periodic for $f$, nor pre-periodic.}

\medskip
In $\Lambda$ all the segment with first coordinate 0 is referred to as the {\em critical set}. This segment has two images by $F$, the point $(-1,y_{-})$ and the point $(1,y_{+})$. The forward orbits of these two points (by iteration of $F$) define the {\em post-critical orbits}. They are denoted by $\CO_{F}^{+}$. 

We emphasize that for $(x,y)\in\Lambda$, the preimages $F^{-n}(x,y)$ can be defined (for every $n$) only for points in $\Lambda\setminus\CO^{+}_{F}$.  

\medskip
\noindent
Our first results deal with hyperbolicity of $F$: 

\medskip
\noindent\noindent
{\bf Theorem A. }{\it For every $(x,y)\in \Lambda\setminus \CO_{F}^{+}$  there exists a  $DF$-invariant unstable direction $E^u(x,y)$.}

\medskip\noindent
As a by-product of Theorem A we recover the next result: 

\medskip\noindent{\bf Corollary B. }{\it Every invariant measure $\mu$ in $\Lambda$ is hyperbolic. The two Lyapunov exponents are $\disp\int\log\left|\frac{\partial g}{\partial y}\right|\,d\mu<0<\int\log\left|f'\right|\circ \pi_{1}\,d\mu\le +\8$, where $\pi_{1}(x,y)=x$. }

Corollary $B$ allows to talk about the stable Lyapunov exponent and the unstable Lyapunov exponent for an invariant measure $\mu$. They will be denoted by $\ls_{\mu}$ and $\lu_{\mu}$. 

\medskip

It is {\it a priori} not forbidden that a measure $\mu$ satisfies $\lu_{\mu}=+\8$. 
Few results exist in the literature concerning measures with infinite Lyapunov exponent. Our strategy is here to split the set of measures; on the one hand   we get the good measures and on the other hand theses measures difficult to control. The definition of good measure follows from the next property of $\Lambda$:

\medskip
\noindent\noindent
{\bf Theorem C. }{\it There exists an increasing sequence of compact sets, $\Lambda_{n}$ such that $\Lambda=\disp\ol{\cup \Lambda_{n}}$. The connected components of $\disp\cup \Lambda_{n}$ are  integral curves of the vector field $E^{u}$.

A point  belongs to $\Lambda_{n}$  if and only if it admits a (local) unstable manifold; this unstable manifold is then the connected component of $\disp\cup \Lambda_{n}$ containing the point.

Moreover, 
 for every ergodic $F$-invariant probability $\mu$, $\disp\mu(\cup\Lambda_{n})=$ 0 or 1. }

\medskip
We refer to Fig. \ref{fig-millefeuille&bananas} for a picture of the long integral curves of the unstable vector field. 

\begin{figure}[htbp]
\begin{center}
\includegraphics[scale=0.3]{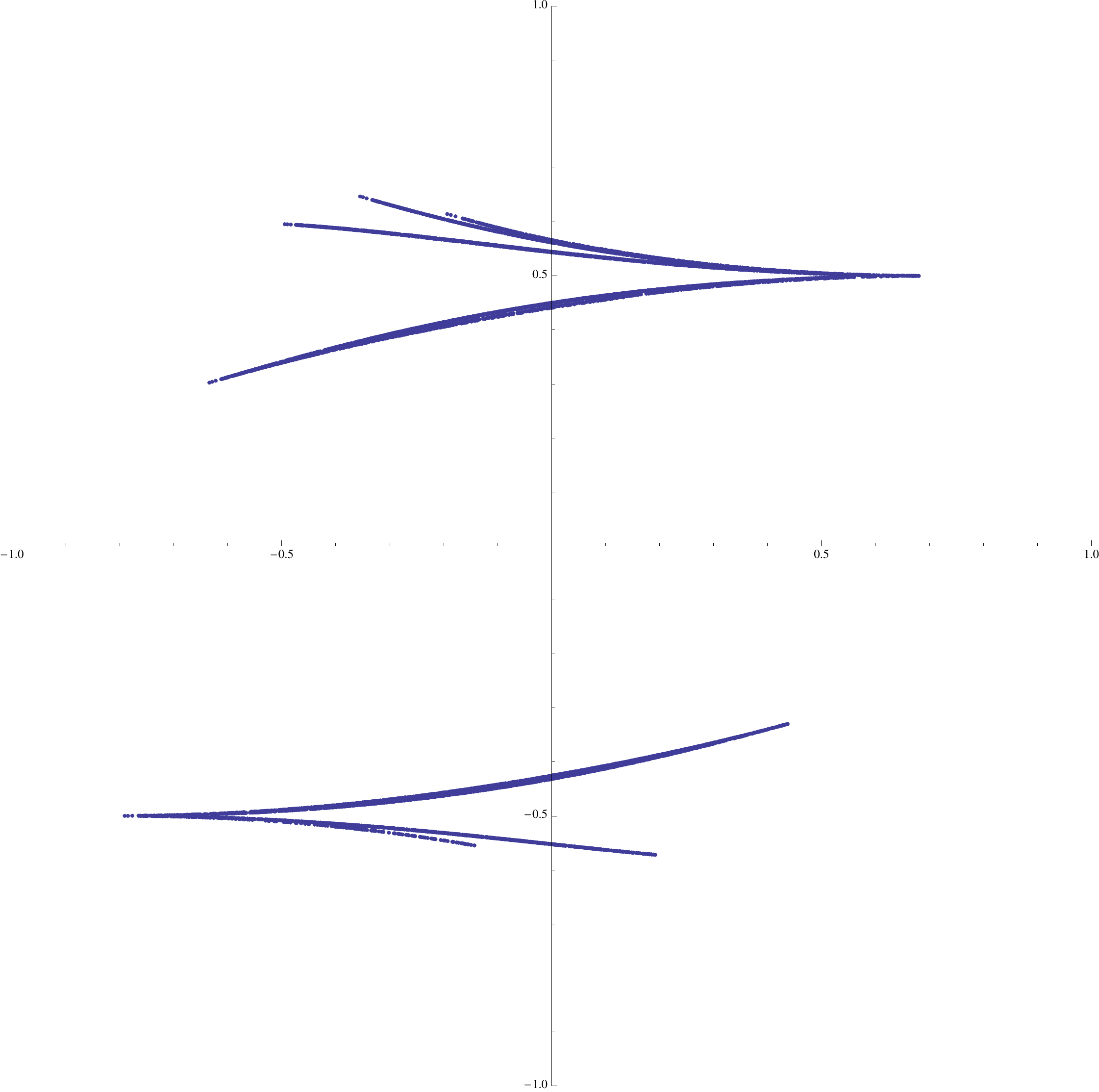}
\caption{The unstable curves and the beaks}
\label{fig-millefeuille&bananas}
\end{center}
\end{figure}

\medskip
\begin{definition}
\label{def-measuregood}
A $F$-invariant ergodic probability $\mu$ is said to be a good measure if $\disp\mu(\cup\Lambda_{n})=1$. 
\end{definition}

\begin{remark}
\label{rem-good-meas}
It is {\it a priori} not forbidden that a good measure has infinite unstable Lyapunov exponent. But due to Pesin Theory, every ``bad'' measure (\ie a measure which is not good) must have infinite unstable Lyapunov exponent.  
$\blacksquare$\end{remark}

Considering  $A_{0}:[-1,1]^{2}\to \R$ H\"older continuous, we recall that an \emph{equilibrium state} for the \emph{potential} $A_{0}$ is any $F$-invariant probability measure $\mu$ such that 
$$h_{\mu}(F)+\int A_{0}\,d\mu=\sup_{\nu}\left\{h_{\nu}(F)+\int A_{0}\,d\nu\right\},$$
where the supremum is taken over $F$-invariant probabilities. The value of the supremum is called the {\em pressure} of $A_{0}$ and shall be denoted by $\CP(A_{0})$. Expansiveness of the one-dimensional dynamics $f$ yields that the supremum is attained (see Corollary \ref{cor-entropy-scs}), hence, it is a maximum. The question is thus to study uniqueness of this equilibrium state. Nevertheless, due to potential existence of ``bad'' measures, even the existence among good measures is not guarantied. We address this problem now:

\begin{definition}
\label{def-relatequil}
A $F$-invariant ergodic good probability measure $\mu$ is said to be a relative equilibrium state among the set of good measures if 
$$h_{\mu}(F)+\int A_{0}\,d\mu=\sup_{\nu\ good}\left\{h_{\nu}(F)+\int A_{0}\,d\nu\right\}.$$
\end{definition}

We give partial answer to the question of uniqueness of the equilibrium state:

\medskip
\noindent\noindent
{\bf Theorem D. }{\it  For every H\"older potential $A_{0}:[-1,1]^{2}\to \R$, there exists at most one relative equilibrium state for $A_{0}$ among the set of good measures. }

\medskip\noindent
We refer the reader to Fig. \ref{fig-3cases} p. \pageref{fig-3cases} for existence or non-existence of the relative equilibrium state.

\subsection{Plan of the paper}

The paper proceeds as follows. In Section \ref{sec-thA} we prove Theorem A and Corollary B. The construction of the unstable direction follows from the construction of a unstable cone fields. 

In Section \ref{sec-thC} we prove Theorem C. As it is said above, we use tools from the Pesin theory to construct a local unstable manifold under the assumption that the point satisfies the \emph{past-stabilization property} (see Def. \ref{def-paststab}). Then we define the sets $\Lambda_{n}$ and prove that the connected component  of $\bigcup \Lambda_{n}$ are ``as long as possible'' unstable manifolds. 

In Section \ref{sec-inducedmap} we construct a nice set with the Markov property. It is called a mille-feuilles. Then, we adapt the theory of local thermodynamic formalism introduced in \cite{leplaideur1} and then developped  in further works of the author. 

In Section \ref{sec-thD} we finish the proof of Theorem D. We define the restricted notion of relative equilibrium state associated to a given mille-feuilles. We prove that the pressure is the relative maximum of the free energies among the good measures. Then we prove that there exists at most a unique equilibrium state and gives condition for existence. 

In Appendix \ref{sec-appen} we discuss some condition where the relative equilibrium state is also a global equilibrium state.


\section{Proofs of Theorem A and Corollary B}\label{sec-thA}


\subsection{Construction of an unstable cones field}
Let $\alpha$ be a positive real number such that
\begin{equation}\label{eq1-condi-alpha}
	\alpha M<\frac12.
\end{equation}
We define the cone field $\CC^u:=\{(u,v)\in \R^2,\ |u|\ge\alpha|v|\}$

\begin{lemma}\label{lem-cone-1}
There exists an integer $N=N(\alpha)$ such that for every $(x,y)\in\Lambda$ there exists $0<n\le N$ satisfying
$$DF^n(x,y).\CC^u\subset \CC^u.$$
\end{lemma}
\begin{proof}
Let $(x,y)$ be in $\Lambda$. We set for every $n\ge0$, $F^n(x,y):=(x_n,y_n)$.
We also set $\disp g_{n,x}:=\frac{\partial g}{\partial x}(x_n,y_n)$ and $\disp g_{n,y}:=\frac{\partial g}{\partial y}(x_n,y_n)$.
We pick $(u,v)$ in $\CC^u$.

The proof is done by induction. We get $\disp DF(x,y).
\left(\begin{array}{c}
	u\\v
\end{array}\right)=\left(\begin{array}{c}
	f'(x_0)u\\g_{0,x}u+g_{0,y}v
\end{array}\right).
$
Assuming the second component of this last vector is not zero the slope of the vector satisfies
\begin{equation}\label{eq1-cones-fields}
	\frac{|f'(x_0)u|}{|g_{0,x}u+g_{0,y}v|}\ge \frac{|f'(x_0)u|}{|g_{0,x}||u|+|g_{0,y}||v|}\ge \frac{|f'(x_0)|}{|g_{0,x}|\alpha+|g_{0,y}|}.\al,
\end{equation}
where we use $\alpha|v|\le |u|$ to get the last inequality.

If
\begin{equation}\label{eq2-cones-fields}
	\frac{|f'(x_0)|}{|g_{0,x}|\alpha+|g_{0,y}|}>1,
\end{equation}
then we get $DF(x,y).\CC^u\subset\CC^u$. If \eqref{eq2-cones-fields} is false, then we compute $DF^2(x,y)\left(\begin{array}{c}
	u\\v
\end{array}\right)=$.

We get
$$DF^2(x,y)\left(\begin{array}{c}
	u\\v
\end{array}\right)=\left(\begin{array}{c}
f'(x_1)	f'(x_0)u\\g_{1,x}f'(x_0)u+g_{1,y}g_{0,x}u+g_{1,y}g_{0,y}v
\end{array}\right).$$
Then we get
$$\frac{|f'(x_1)	 f'(x_0)u|}{|g_{1,x}f'(x_0)u+g_{1,y}g_{0,x}u+g_{1,y}g_{0,y}v|}\ge\frac{|f'(x_1)	 f'(x_0)|\alpha}{\alpha(|g_{1,x}||f'(x_0)|+|g_{1,y}||g_{0,x}|)+|g_{1,y}||g_{0,y}|}.$$

Again, if
\begin{equation}\label{eq3-cones-fields}
	\frac{|f'(x_1)	 f'(x_0)|}{\alpha(|g_{1,x}||f'(x_0)|+|g_{1,y}||g_{0,x}|)+|g_{1,y}||g_{0,y}|}>1,
\end{equation}
then we get $DF^2(x,y).\CC^u\subset \CC^u$.
If \eqref{eq3-cones-fields} is false, we set $R_0:=|g_{0,x}|$ and $R_1:=|g_{1,x}||f'(x_0)|+|g_{1,y}g_{0,x}|$. We use the fact that \eqref{eq2-cones-fields} is false and the bounds $|g_{i,x}|\le M$ and $|g_i,y|\le\frac12$ to obtain

\begin{equation}\label{eq4-cones-fields}
	R_1\le (\frac12+\alpha M)R_0+\frac12 M.
\end{equation}

Now, we iterate our process and compute a lower bound for the slope for $DF^3\left(\begin{array}{c}
	u\\v
\end{array}\right)=:(u_3,v_3)$. Note that \eqref{eq3-cones-fields} is false, thus we get
$$|f'(x_0)f'(x_1)|\le \alpha R_1+|g_{1,y}g_{0,y}|.$$
This yields
$$\frac{|u_3|}{|v_3|}\ge\alpha.\frac{|(f^3)'(x_0)|}{\alpha(|g_{2,y}|R_1+|g_{2,x}|(\alpha R_1+|g_{1,y}g_{0,y}|))+|g_{2,y}g_{1,y}g_{0,y}|}.$$
Setting $R_2:=|g_{2,y}|R_1+|g_{2,x}|(\alpha R_1+|g_{1,y}g_{0,y}|)$, we get
$$R_2\le (\frac12+\alpha.M)R_1+|g_{2,x}||g_{1,y}g_{0,y}|\le (\frac12+\alpha.M)R_1+\frac12M.$$
Iterating this process, as long as we do not get $DF^n(x,y).\left(\begin{array}{c}
	u\\v
\end{array}\right)\in\CC^u$, we can construct a sequence of terms $R_n$, satisfying
$$R_{n+1}\le (\frac12+\alpha.M)R_n+\frac{M}2.$$
Our assumption on $\alpha$ (see \eqref{eq1-condi-alpha}) yields that the sequence is bounded by some constant $K$. Now, we set $DF^n(x,y).(u,v)=:(u_n,v_n)$, and we let the reader check that we have
\begin{equation}
\label{equ1-defconemajor}
\frac{u_n}{v_n}\ge \alpha.\frac{(f^n)'(x_0)}{\alpha R_n+\left(\frac12\right)^n}\ge \alpha.\frac{(\sqrt2)^n}{\alpha K+1}.
\end{equation}
The last term in the right hand side goes to $+\8$ as $n$ increases. This proves that for some $N=N(\alpha)$ it is bigger than 1. The Lemma is proved.
\end{proof}

\begin{remark}
\label{rem-conemajor}
Note that inequality \eqref{equ1-defconemajor} implies that there exists some positive constant $\kappa$ such that for every $(\xi,\eta)\in \Lambda$ and for every $n$, $DF^n(\xi,\eta).\CC^{u}$ is included in the cone 
$$\wh\CC^{u}_{\kappa}:=\{(u,v)\in \R^2,\ |v|\le \kappa|u|\}.$$
$\blacksquare$\end{remark}

\begin{lemma}\label{lem-expansion-cone}
Let $(x,y)$ be in $\Lambda$ and ${\mathbf w}:=(u,v)$ be a vector in $\CC^u$. Then for every $n\ge 0$,
$$||DF^n(x,y).\mathbf{w}||\ge\frac{(\sqrt2)^n}{\sqrt{1+\frac1{\alpha^2}}}||\mathbf{w}||.$$
Moreover, for ${\mathbf w}:=(u,v)$  in $\wh\CC^u_{\kappa}$ and for every $n\ge 0$,
$$||DF^n(x,y).\mathbf{w}||\ge\frac{(\sqrt2)^n}{\sqrt{1+\kappa^2}}||\mathbf{w}||.$$

\end{lemma}
\begin{proof}
We set $DF^n(x,y).\mathbf{w}:=(u_n,v_n)$. As the matrix of $DF^n$ is lower-triangular, we get
$$u_n=(f^n)'(x)u.$$
Moreover the Euclean norm $||DF^n(x,y).\mathbf{w}||$ is larger than $|u_n|$. This yields  $\disp ||DF^n(x,y).\mathbf{w}||\ge (\sqrt2)^n|u|$.

On the other hand, $\mathbf w$ belongs to $\CC^u$, hence $\disp ||\mathbf{w}||\leq \sqrt{1+\frac1{\alpha^2}}|u|$. The first part of the Lemma is proved. The second inequality is obtained in the same way. 
\end{proof}

\begin{lemma}\label{lem-open-cones}
    Let $(x,y)$ be in $\Lambda$ and $n$ be an integer. Then the angle of $DF^n(x,y).\CC^u$ is smaller than $\frac{C}{(2\sqrt2)^n}$ for some $C>0$.
\end{lemma}
\begin{proof}
We consider the two "border" vectors of $\CC^u$, $\mathbf{w_+}:=(1,\alpha)$ and $\mathbf{w_-}:=(1,-\alpha)$.

Note that $\mathbf{w_+}-\mathbf{w_-}=(0,2\alpha)$ and remember that the vertical direction is invariant by $DF$. Hence
$$||DF^n(x,y).(\mathbf{w_+}-\mathbf{w_-})||\le 2\alpha \frac1{2^n}.$$

On the other hand, $\CC^u$ is the set of vectors of the form $a\mathbf{w_+}+b\mathbf{w_-}$, with $ab\ge 0$. Then $DF^n(x,y).\mathbf{w_+}$ and $DF^n(x,y).\mathbf{w_-}$ are border vectors of $DF^n(x,y).\CC^u$. Lemma \ref{lem-expansion-cone} proves that their norm are greater than $\disp \frac{(\sqrt 2)^n}{\sqrt{1+\frac1{\alpha^2}}}$. Then, we use Thales theorem, the fact that the line joining the extremities of $DF^n(x,y).\mathbf{w_+}$ and $DF^n(x,y).\mathbf{w_-}$ is vertical, and that the axe of the cone $DF^{n}(x,y).\CC^{u}$ has slope lower than $\frac1\al$.
\end{proof}

\bigskip
Let us now prove Theorem A. We pick some point $(x,y)$ in $\Lambda\setminus \CO^{+}_{F}$. For a fixed $n$, we can apply Lemma \ref{lem-cone-1}. More precisely, we can apply Lemma \ref{lem-cone-1} by induction to get a sequence of points of the form $(x_{-n_k},y_{-n_k})$, with $0<-n_k+n_{k+1}\le N=N(\alpha)$ and
$$DF^{n_{k+1}-n_k}(x_{-n_{k+1}},y_{-n_{k+1}}).\CC^u\subset \CC^u.$$
Starting from $(x_{-n},y_{-n})$ we construct the sequence by induction and stop when the next point of the sequence is in the forward orbit of $(x,y)$. For the last point of the sequence, say $(x_{-n_0},y_{-n_0})$ we thus get $-n_0\in[-N+1,0]$. In that case we say that \emph{the sequence of hyperbolic jumps starting at $-n$ arrives at $-n_0$} (see Figure \label{fig-sequence cones}).
\begin{figure}[htbp]
\begin{center}
\includegraphics[scale=0.8]{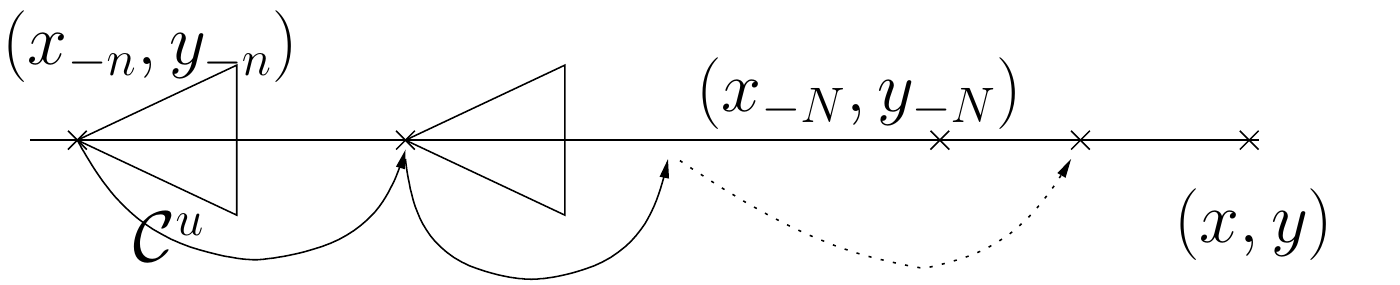}
\caption{Sequence of hyperbolic jumps}
\label{fig-sequence cones}
\end{center}
\end{figure}

Conversely, for $k\in [-N+1,0]$, we denote by $\N(k)$ the set of integers $n$ such that the sequence of hyperbolic jumps starting at $-n$ arrives at $k$.

\paragraph{Claim :} there exists at least one $k$ in $[-N+1,0]$ such that $\N(k)$ is infinite.

\medskip
This claim is a simple consequence of the fact that $\N$ is infinite and $[-N+1,0]$ is finite ! However, we point out that if $n>m$ are in $\N(k)$, this does not necessarily mean that the sequence of jumps starting at $(x_{-n},y_{-n})$ contains the point $(x_{-m},y_{-m})$. This is a source of difficulties we shall control to define the unstable direction. In particular, we emphasize that it is not clear at all that the images of the unstable cones  at $(x_{-k},y_{-k})$ form a decreasing sequence. 

\medskip
We consider some integer $k$ in  $[-N+1,0]$ such that $\N(k)$ is infinite. For $n$ in $\N(k)$ the cone $DF^{n-k}(x_{-n},y_{-n}).\CC^{u}$ seen in the projective space $\mathbb{P}\R^2$ is an interval $[a_{n},a_{n}+a.\theta^{n}]$ with $\theta\le  \frac1{(2\sqrt2)}$ and $a:=\sqrt{1+\frac1{\alpha^2}}$ (see Figure \ref{fig-cone-interval}). 
\begin{figure}[htbp]
\begin{center}
\includegraphics[scale=0.5]{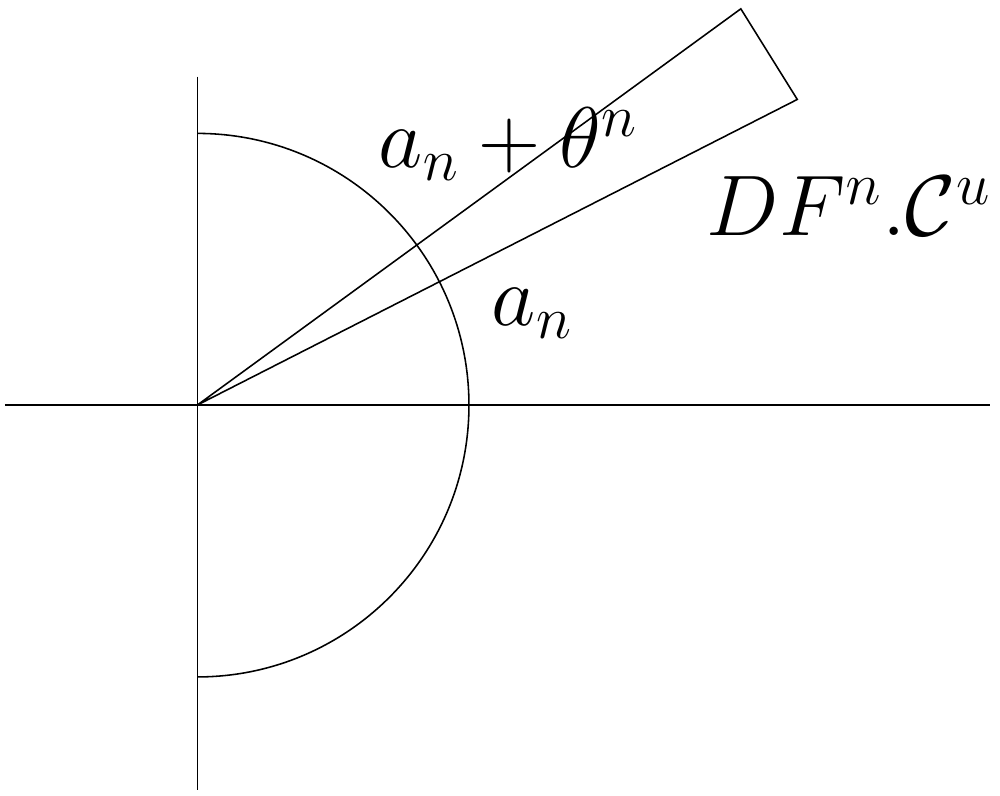}
\caption{Cones in the projective space}
\label{fig-cone-interval}
\end{center}
\end{figure}

We can thus define accumulation points for these segments. Now, we prove that all these accumulation points are equal. 

Note that such an accumulation point is a direction, and can thus been pushed forward in $T_{(x,y)}M=\R^2$. 
We consider two accumulation points, possibly for two different $k$'s in $\mathbb{P}\R^2$, say $\mathbf{e_{1}}$ and $\mathbf{e_{2}}$. We assume that $\mathbf{e_{1}}$ is the accumulation point for a sequence $(m_{j})$ and $\mathbf{e_{2}}$ is the accumulation point for the sequence $(n_{k})$. We pick $n$ very large. The sequence of cones $DF^{l-n}(x_{-l},y_{-l}).\CC^{u}$ converges, for $l=m_{1},m_{2},m_{3},\ldots$ to $DF^{-n}(x,y).\mathbf{e_{1}}$, and   for $l=n_{1},n_{2},n_{3},\ldots$ to $DF^{-n}(x,y).\mathbf{e_{2}}$. 

Note that by Remark \ref{rem-conemajor} both directions $DF^{-n}(x,y).\mathbf{e_{1}}$ and $DF^{-n}(x,y).\mathbf{e_{2}}$ are in the cone $\wh\CC^{u}_{\kappa}$. Choosing two representative vectors of these two directions which coincide on the first coordinate, and then applying $DF^{n}(x_{-n},y_{-n})$ to these two vectors, Lemma \ref{lem-expansion-cone} proves that the distance (in $\P\R^2$) between $\mathbf{e_{1}}$ and $\mathbf{e_{2}}$ is smaller than $\sqrt{1+\kappa^2}\frac1{(2\sqrt2)^n}$. This is true for every $n$, hence $\mathbf{e_{1}}=\mathbf{e_{2}}$.

\subsection{Proof of Corollary B}
Let $\mu$ be a $F$-invariant probability measure. We denote by $\nu$ the push-forward $\pi_{1,*}\mu$. It is a $f$-invariant probability measure on $[-1,1]$. 

\begin{lemma}
\label{lem-0-0mesure}
The measure $\nu$ satisfies $\nu(\{0\})=0$.
\end{lemma}
\begin{proof}
Assume by contradiction that $\nu(\{0\})>0$. Then, for every $\eps>0$, $\nu(]-\eps,\eps[)>0$. For $\eps$ sufficiently small, $f$ is a bijection from $]-\eps,0^{-}]$ onto $]1-\rho(\eps),1]$ and from $[0^{+},\eps[$ onto $[-1,-1+\rho'(\eps)[$, with $\rho(\eps)$ and $\rho'(\eps)$ going to 0 as $\eps\to 0$. 

One (at least) of these two intervals has a $\nu$-measure bigger than $\disp\frac{\nu(0)}2$ for every $\eps$ (recall that $\nu$ is $f$-invariant). Say it is $]1-\rho(\eps),1]$. Doing $\eps\to0$ this shows $\nu(\{1\})>0$. Now, $\nu(\{f(1)\})$ is bigger than $\nu(\{1\})$ because $\nu$ is $f$-invariant. Doing this by induction, we prove $\nu(f^{k}(1))\ge \nu(1)>0$ for every $k$, and  $\nu$ should be infinite. This contradicts the fact that it is a probability.
\end{proof}

\begin{corollary}
\label{cor-entropy-scs}
The metric entropy $\mu\mapsto h_{\mu}$ is upper-semi continuous. 
\end{corollary}
\begin{proof}
There is a canonical semi-conjugacy between a subshift in $\{0,1\}^{\N}$ and $([-1,1],f)$. For $x$ which is not a preimage of 0, $f^{k}(x)$ always belongs either to $[-1,0[$ or  to $]0,1]$. We associate to $x$ the itinerary coding, assigning 0 if $f^{k}(x)$ belongs to $[-1,0[$ and $1$ if it belongs to $]0,1]$. 

Itinerary of -1 and 1 are well defined because $0$ is not periodic. Then, we associate to 0 two codes, one which is 0-and the coding sequence of -1 and one which is 1 and the coding sequence of 1. 
 
If $x$ is a preimage of 0, its coding sequence is well determined up to it ``falls'' on 0. Then we just concatenate  any of the coding sequence of 0.

This yields a set of admissible codes say $\K\subset\{0,1\}^{\N}$. The inclusion is strict because, \eg, the codes $000\ldots$ or $111\ldots$ are not eligible. 
As $f$ is expanding,  the map which assigns to a code in $\K$ the associated point is one-to-one except on the codes mapped on preimages of $0\in [-1,1]$; it is actually two-to-one for preimages of 0 (including 0). Expansion also shows that the map is continuous, actually H\"older continuous.

The main consequence of Lemma \ref{lem-0-0mesure} is that any $f$-invariant probability can be lifted into $\K$ by a $\s$-invariant probability. Upper semi-coninuity of $h_{\mu}$ is $\{0,1\}^{\N}$ yields upper semi-continuity in $[-1,1]$. 
\end{proof}

Note that by Lemma  \ref{lem-0-0mesure}, for any $\mu$ $F$-invariant, $\mu$-almost every $(x,y)\in\Lambda$ has an unstable direction $E^{u}(x,y)$.

On the other hand, for such a point $(x,y)$,
$DF^{n}(x,y)$ is of the form 
\begin{equation}
\label{equ1-form-differentiel}
\left(\begin{array}{cc}(f^n)'(x) & 0 \\U_{n} & \disp\deriv{g^n}{y}\end{array}\right),
\end{equation}
where $\disp |U_{n}|\le C.\sum_{k=0}^{n-1}|(f^k)'(x)|$ for some constant $C$. This implies that the vertical direction is invariant. The asymptotic logarithmic expansion in that direction is  $\disp\int\log\left|\deriv{g}{y}\right|\,d\mu$. It is a negative number.

Moreover, computing the largest Lyapunov exponent $\lu_{\mu}$, we get for \emph{a.e.} $(x,y)$
$$\lu_{\mu}:=\lim_{\ninf}\frac1n\log||DF^{n}(x,y)||,$$
and \eqref{equ1-form-differentiel}  shows that $\disp\lu_{\mu}=\lim_{\ninf}\frac1n\log|(f^n)'(x)|$. By Birkhoff theorem, this last limit is $\disp\int\log|f'|\,d\nu$. 
Hence, Osseledec theorem insures that there exists a $DF$-invariant direction associated to this positive Lyapunov exponent. This direction must coincide with $E^{u}(x,y)$ (for $\mu$-\emph{a.e.} $(x,y)$) otherwise we would get three invariant directions in a 2-dimension space.

\section{Proof of Theorem C}\label{sec-thC}

\subsection{Construction of the unstable manifold}

Let $(x,y)$ be in $\Lambda$. For simplicity, we set $F^{-n}(x,y):=(x_{-n},y_{-n})$.

Pick an interval $I$ in $[-1,1]$. We consider the sequence of iterates of the interval, $f(I), f^{2}(I), \ldots, f^{k}(I)$. As long as these image intervals do not contain zero their image are still an interval (whose length increases in $k$). Hence, there eventually exists some minimal $k$ such that $f^{k}(I)\ni 0$. The image $f^{k+1}(I)$ is a union of two intervals, and we say that \emph{the interval is cut}.

\noindent{\bf Notation}: for simplicity we shall consider intervals $[x_{-n},0[$ but they could also denote $]0,x_{-n}]$ if $x_{-n}>0$.


\begin{definition}
\label{def-paststab}
Let $(x,y)$ be a point in $\Lambda\setminus \CO^{+}_{F}$.  
We say that a point $(x,y)$ has the past-stabilization property if 
there exist $\delta=\delta(x,y)>0$ and $N=N(x,y)$ such that for every $n\ge N$, 
 if  the iterates of interval, say  $f^{k}([x_{-n},0[)$ are never cuts, for $k\le n$, then the length of $f^{n}([x_{-n},0])$ is bigger than $\disp\frac\delta2$. 
\end{definition}

\paragraph{\bf Notation.} For $M=(x,y)$ in $\Lambda\setminus\CO^{+}_{F}$,  we call the inverse branches of $x$ following $M$ the sequence of points $(x_{-n})$ in the interval $[-1,1]$. They satisfied $f^{n}(x_{-n})=x$.

\medskip
We left it to the reader to check the following equivalent definition for the past-stabilization property:

\begin{definition}
\label{def-paststab2}
The point $M=(x,y)$ has the past-stabilization property if there exists some $\eps>0$ such that  the preimages of the interval $]x-\eps,x+\eps[$ following the inverse branches defined by $F^{-n}(x,y)$ are intervals which never contain 0.
\end{definition}

\begin{remark}
\label{rem-paststab}
This later definition shows that if $M$ has the past-stabilization property, then it also holds for $F^{-k}(M)$, $k\in N$.  
$\blacksquare$\end{remark}

If $\delta$ is a positive real number and $z$ in $[-1,1]$, $||z||_{\delta}$ denotes $\max(1,\disp\frac{|z|}{\delta})$.
\begin{definition}
\label{def-bsrp}
Let $(x,y)$ be a point in $\Lambda\setminus\CO^{+}_{F}$.  
We say that a point $(x,y)$ has the  backward slow recurrence property if for every $\eps>0$, exists $\delta>0$ such that 
$$\limsup_{\ninf}\frac1n\sum_{k=0}^{n-1}\left|\log ||\pi_{1}\circ F^{-k}(x,y)||_{\delta}\right|\le \eps.$$
In abridge way we shall say $(x,y)$ satisfies the b.s.r.p. or equivalently $(x,y)$ is backward slow recurrent (b.s.r. in abridge way).
\end{definition}

\begin{lemma}
\label{lem-bsrp-dilatboule}
Assume that $(x,y)$ is {\em b.s.r.}. Then, it has the past-stabilization property. \end{lemma}
\begin{proof}
Pick $0<\eps<\frac14\log2$ and then pick $\delta$ associated as in Definition \ref{def-bsrp}.

Assume the Lemma is false. 
Then, for every $N$, there exists $n\ge N$  such that the interval $[x_{-n},0[$ or $]0,x_{-n}]$ is never cut and 
\begin{equation}
\label{equ1-bsrp-dilatboule}
|f^{n}(x_{-n})-f^{n}(0)|\le \frac\delta2.
\end{equation}
Here $f(0)$ is  either 1 or -1 depending on $x_{-n}$'s sign.

As $f$ is expanding, the last inequality certainly implies $|x_{-n}|<\frac\delta2(\sqrt2)^{n}$. However, we want a thinner estimate for $|x_{-n}|$.

We can always assume that $\delta$ is sufficiently small such that 
$$\forall \, z, |z|\le \frac\delta2\Longrightarrow|f'(z)|\ge \frac{C}{|z|^{1-\rho}},$$
with $C>0$ and $\rho\in ]0,1[$. 

Moreover, the norm of derivative is decreasing with respect to the distance to 0. This means that for every $z$ in $[x_{-n},0[$ (or $]0,x_{-n}]$), 
$$|f'(z)|\ge |f'(x_{-n})|.$$
For the rest of the orbit, we simply use expansion. Then, \eqref{equ1-bsrp-dilatboule} yields,
$$(\sqrt{2})^{n-1}C|x_{-n}|^{\rho}\le \frac\delta2,$$
which is equivalent to 
$$\frac1{n+1}\left|\log \frac{|x_{-n}|}{\delta}\right|\ge \frac{n-1}{2(n+1)\rho}\log2-\frac1{(n+1)\rho}\log\frac{\delta^{1-\rho}}{2}-\frac1{(n+1)\rho}\log C.$$
Note that $\disp \frac{|x_{-n}|}{\delta}=||x_{-n}||_{\delta}$. Then we get
$$\frac1{n+1}\sum_{k=0}^{n}\left|\log||x_{k}||_{\delta}\right|\ge \frac{n-1}{(n+1)\rho}\eps-\frac1{(n+1)\rho}\log\frac{\delta^{1-\rho}}{2}-\frac1{n+1}\log C.$$
As $n$ is as big as wanted, this is in contradiction with the {\em b.s.r.p.}.
\end{proof}

\begin{theorem}
\label{theo-var-instable}
If $(x,y)$ has the past-stabilization property, then it admits a local unstable manifold, $W^{u}_{loc}(x,y)$;  $W^{u}_{loc}(x,y)$ satisfies the following properties:
\begin{enumerate}
\item $W^{u}_{loc}(x,y)$ is a $\frac1\alpha$-Lipschitz continuous graph over an horizontal interval $]x-\eta(x,y)/2,x+\eta(x,y)/2[$ for some $\eta(x,y)>0$, and $\alpha$ is defined in Inequality \eqref{eq1-condi-alpha}.
\item The function $(x,y)\mapsto \eta(x,y)$ is Borel. 
\item For every $n$, $\pi_{1}\circ F^{-n}(W^{u}_{loc}(x,y))\not\ni0$ and $\eta(x,y)$ is maximal with this property. 
\end{enumerate}
\end{theorem}

\begin{remark}
\label{rem-lambdaWu}
The Theorem shows that if $M=(x,y)$ belongs to $\Lambda\setminus\CO^{+}_{F}$, then $W^{u}_{loc}(x,y)\subset \Lambda\setminus\CO^{+}_{F}$. 
$\blacksquare$\end{remark}

%
%
%

\begin{proofof}{Theorem \ref{theo-var-instable}}
 We consider $\delta=\delta(x,y)$ as in Lemma \ref{lem-bsrp-dilatboule}. For every $n$, we consider the horizontal segment $]x_{-n}-\frac\delta2,x_{-n}+\frac\delta2[\times\{y_{-n}\}$.

 We study each half of this interval in a similar way. 
 The map $F$ is expansive. The lengths of intervals $F^{k}\left(\left[x_{-n},x_{-n}+\frac\delta2\right[\right)$  increase in $k$, unless one interval is ``cut''  by the discontinuity at 0 (\ie one of the interval contains 0). 
 If the intervals are not cut, expansivity shows that the image interval overlaps all the half-ball $[x,x+\frac\delta2[$.

If one interval is cut, say at some iterates $k$, then the new interval to consider is of the form $[x_{-k},0[$ or $]0,x_{-k}]$. If the images of this interval are never cut, then 
Lemma \ref{lem-bsrp-dilatboule} shows that if $k>N(x,y)$ the final image interval has length bigger than $\frac\delta2$. If not, we induces this process until we get an interval say $[x_{-j},0[$ with $j\le N(x,y)$. But then, we only have to consider a finite number of intervals, and can consider the intersection of the connected components  of  the $f^{j}([x_{-j},0[)$'s which contain $x$.

We do this for both side (the left hand side and the right hand side) of $x$ considering the $F^{k}\left(\left[x_{-n}-\frac\delta2,x_{-n}\right]\right)$'s.

This gives some positive $\eta$ such that for every $n$, 
the connected component  which contains $(x,y)$ of the image by $F^{n}$ of $]x_{-n}-\frac\delta2,x_{-n}+\frac\delta2[\times\{y_{-n}\}$ is a graph (at least) over the  interval $]x-\frac\eta2,x+\frac\eta2[$, say  of a map $\varphi_{n}$.

Contraction in the vertical direction, and the fact that the map $F$ has a fibered structure ($F(x,y)=(f(x),g(x,y))$ and thus maps verticals into verticals) show that the sequence of graphs $(\varphi_{n})$ converge uniformly to a limit graph, say $\varphi_{u,x,y}$

Uniform expansion and contraction yield (this is a standard computation) that these maps $\varphi_{n}$ are all $\frac1\alpha$-Lipschitz continuous because they must have slope in the unstable cone field.

 It remains to prove that $\eta$ is a Borel map. This follows from the fact that for each $(x,y)$, there exists $n$, such that 
$\eta/2=\min_{k\le n}\left\{\left|f^{k}(x_{-k})-f^{k}(0)\right|\right\}$.

By construction, for every $n$, $\pi_{1}\circ F^{-n}(graph(\varphi_{u,x,y}))$ is never cut by iterating $f^{k}$ ($k\le n$). 
We can then consider the maximal $\eta$ with this property. It is again a Borel map. 
\end{proofof}

By construction,  the $\eta$ is maximal. This means that on the left or on the right there is a cutting point, \ie a point whose first coordinate is an image of 0. At that point the unstable leaf stops or has a ``{\it beak}''.

\subsection{Proof of Theorem C}

We define $\Lambda_{n}$ as the set of points  $M=(x,y)$ in $\Lambda$ whose unstable local manifold contains a graph over the interval $]x-\frac1n,x+\frac1n[$. 

\begin{lemma}
\label{lem-lambdancompact}
The set $\Lambda_{n}$ is compact. 
\end{lemma}

\begin{proof}
It is sufficient to prove that $\Lambda_{n}$ is closed. 
We consider a sequence of points $M_{k}:=(\xi_{k},\xi'_{k})$ converging to $(x,y)$ and all in $\Lambda_{n}$. Fix $\eps$ very small. There exists $k_{\eps}$ such that for every $k\ge k_{\eps}$, $|x-\xi_{k}|<\eps$. Then, for $k\ge k_{\eps}$, $W^{u}_{loc}(\xi_{k},\xi'_{k})$ contains a graph over the interval $]x-\frac1n+\eps, x+\frac1n-\eps[\subset ]\xi_{k}-\frac1n,\xi_{k}+\frac1n[$. 

Fix some integer $j$ bigger than $k_{\eps}$. Then, consider $k$ sufficiently big such that the inverse branches of $F$ for $M$ and for $M_{k}$ coincide at least until $j$. The interval $]x-\frac1n+\eps,x+\frac1n-\eps[$ is never cut by $f^{-i}$, $i\le j$ (following the $F^{-i}(M)$'s) because it is contained in the interval $]\xi_{k}-\frac1n,\xi_{k}+\frac1n[$ which is never cut following the $F^{-i}(M_{k})$'s and these inverse branches coincide.

As $j$ can be chosen as big as wanted, Definition \ref{def-paststab2} shows that $M$ has the past stabilization property. 
Furthermore, the local unstable manifold $W^{u}_{loc}(M)$ contains a graph over the interval $]x-\frac1n+\eps,x+\frac1n-\eps[$. Letting $\eps\to 0$ this proves that $M$ belongs to $\Lambda_{n}$. 
\end{proof}

\begin{proposition}
\label{prop-cuplambdandense}
The set $\disp\cup\Lambda_{n}$ is dense in $\Lambda$.
\end{proposition}
\begin{proof}
Let $M=(x,y)$ in $\disp\Lambda\setminus\cup\Lambda_{n}$. Theorem \ref{theo-var-instable} shows that $M$ cannot get the past-stabilization property. Thus, there exists an increasing sequence $(k_{n})$ such that for each $n$, the intervals $f^{j}([x_{-k_{n}},0[)$ are never cut (for $0\le j\le k_{n}$) and the sequences $\disp \diam(f^{k_{n}}([x_{-k_{n}},0[))$ decreases\footnote{Following the above proofs, for simplicity we consider the $x_{-k_{n}}$ are negative} to $0$. 

It is known (see \cite{viana}) that there exists a {\em SRB} measure for $F$. The projection on $[-1,1]$ of its support is the whole interval $[-1,1]$. Therefore, for every $n$, there exists some point $Q_{n}=(\xi_{n},\eta_{n})$, generic for this {\em SRB}-measure such that $x_{-k_{n}}<\xi_{n}<0$. 

The point $Q_{n}$ is ``generic'' for the measure, and we can assume it has a local unstable manifold $W^{u}_{loc}(Q_{n})$ (in the sense of Pesin theory). Moreover, for every $k$, $F^{k}(Q)$ belongs to some $\Lambda_{j}$. 
By definition of $k_{n}$, $F^{k_{n}}(Q_{n})$ belongs to the vertical band $\pi^{-1}(f^{k_{n}}([x_{-k_{n}},0[))$. It also belongs to the horizontal stripe $[-1,1]\times [y-2^{\frac{k_{n}}2},y+2^{\frac{k_{n}}2}]$. 

Therefore the sequence  of points $F^{k_{n}}(Q_{n})$  converges to $M$ as $n$ goes to $+\8$. 
\end{proof}

\begin{definition}
\label{def-wuLocgrand}
We call $u$-curve in $[-1,1]$ any integration curve for the vector field $E^{u}$: $\CF^{u}$ is an $u$-curve if and only if for every $M\in \CF^{u}$, $T_{M}\CF^{u}=E^{u}(M)$. 
\end{definition}

\begin{remark}
\label{rem-nobeakucurve}
Beaks do not belong to $u$-curves. 
$\blacksquare$\end{remark}

Existence of $u$-curves directly follows from Theorem \ref{theo-var-instable}.  Note that inclusion defines a relative order relation between $u$-curves. Then,  Zorn's lemma shows that  maximal element exist. For $M$ in $\cup\Lambda_{n}$ we write $W^{u}_{L}(M)$ a maximal (for the inclusion) $u$-curve which contains $M$. 

\begin{lemma}
\label{lem-Wumaxunic}
For every $M$ in $\cup\Lambda_{n}$ there exists a unique maximal $u$-curve $W^{u}_{L}(M)$. It is called the maximal local unstable manifold for $M$. Moreover, for every $j\ge 0$, $F^{-j}(W^{u}_{L}(M))$ is well defined, connected and does not intersect the critical set. 
\end{lemma}

\begin{proof}
Let $M=(x,y)$ in $\cup\Lambda_{n}$. Let $W^{u}_{L}(M)$  be any maximal $u$-curve containing $M$. By definition any point in $W^{u}_{L}(M)$ belongs to $\Lambda\setminus\CO^{+}_{F}$. Then, $F^{-j}(W^{u}_{L}(M))$ is well-defined and does not intersect the critical set. It is connected because otherwise $W^{u}_{L}(M)$ would contain a beak, hence would intersect $\CO^{+}_{F}$.

Assume that there exists two maximal $u$-curves, say $W^{u}_{1}$ and $W^{u}_{2}$ containing $M=(x,y)$. 
Both $W^{u}_{1}$ and $W^{u}_{2}$ are graphs over two intervals $I_{1}$ and $I_{2}$ containing $x$ in their interior.  We set $W_{i}:=graph(g_{i})$. 
Both graphs coincide for some interval ( {\it a priori} possibly containing only $M$) say $I$. The left extremal point of $I$ is denoted by $x_{l}$ and the right extremal point is denoted by $x_{r}$. By continuity we can define $g_{i}(x_{j})$ with $i=1,2$ and $j=l,r$. Necessarily $g_{1}(x_{l})=g_{2}(x_{l})$ and $g_{1}(x_{r})=g_{2}(x_{r})$. 

Assume $M_{l}:=(x_{l},g_{1}(x_{l}))$ is in $\bigcup\Lambda_{n}$. There are only three possibilities:
\begin{itemize}
\item $M_{l}$ is a beak. 
\item $M_{l}$ is a limit point just for one $W^{u}_{i}(M)$.
\item $M_{l}$ is not a limit point and $W^{u}_{1}$ and $W^{u}_{2}$ split at that moment and go further.
\end{itemize}

We claim that only the first alternative is possible. 
Indeed, the second would contradict maximality for the shortest $W^{u}_{i}$ (the one which stopped at $M_{l}$), because it could be continued with a piece of the other one. 
The third alternative is also impossible for the following reason. Both $W^{u}_{i}(M)$ split at $M_{l}$ and continue a little bit further on the left. As they have bounded slope, they necessarily exits $x'<x_{l}$ such that $g_{1}(x')$ and $g_{2}(x')$ exist and are different. 
Now, $F^{-j}(W^{u}_{1})(M)$ and $F^{-j}(W^{u}_{2})(M)$ are connected; they must  be in the same connected component of $F([-1,1]^{2})$ because they have an non-empty intersection. 
On the other hand,  $F^{-j}$ expands in the vertical direction. Thus, there exists some $j$ such that $F^{-j}((x',g_{1}(x')))$ and $F^{-j}((x',g_{2}(x')))$ are in two different connected component of $F([-1,1]^{2})$. This is a contradiction with connectivity.

The same reasoning shows that the other extremal point $M_{r}:=(x_{r},g_{1}(x_{r}))$ is also a beak. 
Consequently $W^{u}_{1}=W^{u}_{2}$ because they coincide into the ``interval'' $]M_{l},M_{r}[$ and these two extremal points are not in $\bigcup\Lambda_{n}$. 
\end{proof}

\begin{remark}
\label{rem-Wuopen}
Actually, adapting the proof of Lemma \ref{lem-Wumaxunic} we get that $W^{u}_{L}(M)$ is a graph over an interval. The two extreme points of $W^{u}_{L}(M)$ (meaning with extreme $x$-coordinates) are beaks and $W^{u}_{L}(M)$ is an open embedded submanifold. 
$\blacksquare$\end{remark}

We define the relation $M\sim_{u}M'\iff W^{u}_{L}(M)=W^{u}_{L}(M')$.

\begin{proposition}
\label{prop-classe-WuL}
The equivalent classes for $\sim_{u}$ are the connected components of $\cup\Lambda_{n}$.
\end{proposition}
\begin{proof}
An equivalence class is a maximal $u$-curve. It is thus connected. Moreover $\disp\cup\Lambda_{n}$ is included into $\disp \bigcup_{M\in \cup\Lambda_{n}}W^{u}_{L}(M)$. 

It remains to prove that every $W^{u}_{L}(M)$ is included in $\cup\Lambda_{n}$ and that they indeed are the connected components of $\cup\Lambda_{n}$. 

\medskip
Let $M'$ be in some $W^{u}_{L}(M)$. We will show that $M'$ has the past-stabilization property. The set $W^{u}_{L}(M)$ is connected  does not contain points of $\CO^{+}_{F}$. Then for every integer $k$, $F^{-k}(W^{u}_{L}(M))$ is a connected curve. It contains $F^{-k}(M')$. Consequently, $\pi_{1}(F^{-k}(W^{u}_{L}(M)))$ is an interval in $[-1,1]$ which is never cut by $f^{j}$, $0\le j\le k$ which contains $x'_{k}:=\pi_{1}(F^{-k}(M'))$. This shows that $M'$ has the past stabilization property.

Therefore we get $\disp \cup\Lambda_{n}=\cup_{M\in\cup\Lambda_{n}} W^{u}_{L}(M)$.

\medskip
Consider $M \in \cup\Lambda_{n}$ and $W$ the connected component of $\cup\Lambda_{n}$ which contains $M$. 
As $W^{u}_{L}(M)$ is connected, $W^{u}_{L}(M)\subset W$. Now assume that there exists $M'\in W\setminus W^{u}_{L}(M)$ and $M''\in W^{u}_{L}(M)$ such that $\pi_{1}(M')=\pi_{1}(M'')$. Again, for every $k$, $F^{-k}(W)$ is a connected set (otherwise there would be a beak in $W$). But $F^{-1}$ expands in the vertical direction and $F^{-k}(M')$ and $F^{-k}(M'')$ cannot always be in the same connected component of $F([-1,1]^{2})$. 

This shows that if $W$ is (or contains) an union of different maximal unstable local manifolds, these manifolds cannot overlaps in the vertical direction. 
On the other hand, $\pi_{1}$ is a continuous function and $\pi_{1}(W)$ has to be a connected set (thus an interval) in $[-1,1]$. Nevertheless if $W$ contains several maximal local unstable manifolds, their projections on $[-1,1]$ are disjoint intervals because the manifolds cannot overlaps (in $x$-direction). Over a junction point, there should be another maximal local unstable manifold and there should be overlapping. On the other hand, connectedness shows that an open interval cannot be written as the union of disjoints intervals.  This proves that $W$ cannot contain more than one maximal local unstable manifold, hence $W=W^{u}_{L}(M)$. 
\end{proof}

\begin{proposition}
\label{prop-cuplambdanergo}
Let $\mu$ be a $F$-invariant ergodic measure. Assume $\mu(\Lambda_{j})>0$ for some integer $j$. Then $\mu(\cup\Lambda_{n})=1$. 
\end{proposition}
\begin{proof}
We claim that for $\mu$ a.e. $M$, $F(M)$ belongs to some $\Lambda_{j}$. Actually, if $M$ belongs to $\Lambda_{n}$ and $\pi_{1}(F(M))\neq 0$, then $F(M)$ again belongs to some $\Lambda_{j}$ because it has the past-stabilization property. 
Conversely, if $\pi_{1}(F(M))=0$, $F(M)$ clearly does not belong to any $\Lambda_{j}$.

Now Lemma \ref{lem-0-0mesure} shows that $\mu(\pi_{1}^{-1}(0))=0$ and then $\mu(F^{-1}(\pi_{1}^{-1}(0)))=0$ (note that $F^{-1}(\pi_{1}^{-1}(0))$ is well defined). 

This shows that the following inclusion holds $\mu$-a.e. :
$$F(\cup\Lambda_{n})\subset \cup\Lambda_{n}.$$
As the measure $\mu$ is ergodic and $\mu(\Lambda_{j})>0$, then $\mu(\cup\Lambda_{n})=1$. 
\end{proof}


\section{Rectangle and induced sub-system}\label{sec-inducedmap}
\subsection{mille-feuilles}

The periodic orbits for the one-dimensional system $([-1,1],f)$ are dense. We pick some positive $\wh\delta$ very small such that the set of points $(x,y)$ satisfying 
\begin{equation}\label{def-delta-band}
\eta(x,y)>2\wh\delta
\end{equation}

is non-empty\footnote{Remind $\eta(x,y)$ was defined in Theorem \ref{theo-var-instable}.} . We also assume  
\begin{equation}
\label{equ2-condi-deltabanda}
f(\wh\delta)<-\wh\delta \text{ and }f(-\wh\delta)>\wh\delta.
\end{equation}
Then, we pick a $\wh\delta$-dense periodic orbit for $([-1,1],f)$, and consider two consecutive points of the orbit (for the relation of order $\le$) in $[-1,0[$ or $]0,1]$. We call them $P_{l}<P_{r}$. Remember that $0$ is neither periodic (for $f$) nor pre-periodic. Therefore, the periodic orbit $P_{l}, f(P_{l}), \ldots, P_{r},f(P_{r}),\ldots$ does not belong to the forward orbit of $0$. 

We consider the vertical  band $\CB$ in the two dimension system 
$$\CB=\left\{(x,y),\ P_{l}\le x\le P_{r},\ -1\le y\le 1\right\}.$$
The \emph{interior of the band} denotes points with in $\CB$ with first coordinate in $]P_{l},P_{r}[$
Then, we consider all the pieces of local unstable manifolds going as a graph over $[P_{l},P_{r}]$ from the left border of $\CB$ to the right border of $\CB$. We only consider pieces of manifolds without ``beak'' over  the whole interval $[P_{l},P_{r}]$. We get a set called mille-feuilles  and denoted $\CM_{0}$. By construction, the set is laminated. Each element of the lamination is called an \emph{ (unstable) leaf} of the mille-feuilles. 

\begin{lemma}
\label{lem-millefeuille-ferme}
The mille-feuilles $\CM_{0}$ is closed. 
\end{lemma}
 \begin{proof}
We consider a family of manifolds converging. All these manifolds are $\frac1\alpha$-Lipschitz continuous graphs over $[P_{l},P_{r}]$. As they form a lamination, up to a subsequence, we can always assume that the sequence is monotone. It converges to a $\frac1\alpha$- Lipschitz graph over $[P_{l},P_{g}]$. 

It could happen that every (or only  infinitely many) graph in the monotone considered sequence has a beak on one side of the band $\CB$, and that this sequence of beaks converges to a point. The limit point cannot  be a beak and belong to the boundaries of the band, \ie to the vertical over $P_{l}$ or $P_{r}$ because the periodic orbit $P_{l}, f(P_{l}),\ldots$ does not belong to the forward orbit of 0. 

The unique other possibility  is that no beaks converge to the borders, thus they stay at a positive distance to the borders.

Therefore, the limit graph is admissible and $\CM_{0}$ is closed. 
\end{proof}

\begin{definition}
\label{def-vertiband}
Let $(x,y)$ be in the interior of the band $\CB$ and in $\CM_{0}$. Assume that $F^{n}(x,y)$ also belongs to the interior of $\CB$ and to $\CM_{0}$ ($n>1$). Then the connected component of $F^{-n}(\CB)\cap \CB$ which contains $(x,y)$ is called the vertical band associated to $(x,y)$ and time $n$
\end{definition}

We left it to the reader to check that contraction in the vertical direction shows that the connected component of $F^{-n}(\CB)\cap \CB$ which contains $(x,y)$  is indeed a vertical band of the form $[a,b]\times [-1,1]$.

 \begin{lemma}
\label{lem-nobeaks}
Let $(x,y)$ be in the interior of the band $\CB$ and in $\CM_{0}$. Let $\CB':=\disp [a,b]\times [-1,1]$ be the vertical band associated to $(x,y)$ and time $n$. Then for every $(\xi,\zeta)$ in $\CB'\cap \CM_{0}$, $F^{n}(W^{u}_{loc}(\xi,\zeta)\cap  \CB')$ is a complete unstable leaf of the mille-feuilles $\CM_{0}$
\end{lemma}
\begin{proof}
Consider $(\xi,\zeta)\in \CB'\cap\CM_{0}$ and the associated piece of unstable leaf intersected with the mille-feuilles
$W:=W^{u}_{loc}(\xi,\zeta)\cap \CB'$. This piece of unstable leaf has no beaks. Remember that $F$ lets verticals invariant  and contracts them. Therefore, $F^{n}(\xi,\zeta)$ also belongs to $\CB$. The image $F^{n}(W)$ is a one-dimensional manifold which joins the left hand side of $\CB$ to the right hand side. 

We want to check that $F^{n}(W) $ is a leaf of the mille-feuilles $\CM_{0}$. This holds if and only if $F^{n}(W)$ is a piece of local unstable manifold and has no beaks (even on the borders).
By construction, it is a piece of global unstable manifold. We thus just have to check that it contains no beaks.

Assume, by contradiction, it has some beak over the interval $[P_{l},P_{r}]$. This beaks cannot be the image of some ``previous'' beak in $W$. Indeed, the dynamics expands the unstable direction (unless it creates discontinuity) and by definition there were no beaks in $W$. \emph{A fortiori} there is no beak in $W\cap \CB'$. 

This shows that the unique possibility is that the beak is created in $W\cap \CB'$ while taking the image by some $F^{j}$ with $0\le j\le n-1$. Seeing this in the one-dimensional system $([-1,1],f)$, this means that $F^{j}([a,b])$ contains $0$. This would also create a beak in the whole vertical, thus in $W^{u}_{loc}(x,y)\cap \CB'$ too. This would be in contradiction with  $F^{n}(x,y)\in \CM_{0}$.  
\end{proof}

\begin{remark}
\label{rem-propmarkov}
Due to Lemma \ref{lem-nobeaks}, we say that the mille-feuilles has the Markov property. 
$\blacksquare$\end{remark}

This last lemma shows that the first-return map into $\CM_{0}$ has good dynamical properties inherited from the Markov property: considering $(x,y)\CM_{0}$ and in the interior of $\CB$, we say that $n$ is the \emph{first return-time} into $\CM_{0}$ if $F^{n}(x,y)\in \CM_{0}$ and 
\begin{description}
\item[ $\bullet$] either $F^{n}(x,y)$ belongs to the interior of $\CB$,
\item[$\bullet$] or it is in the border for the topology in the unstable manifolds of such points. This means that $F^{n}(x,y')$ is accumulated by points $F^{n}(x',y')\in \CM_{0}$ which are in the interior of $\CB$ and such that $(x',y')\in W^{u}_{loc}(x,y)\cap \CM_{0}$. 
\end{description}

For the rest of the paper we denote by $\tau(x,y)$ the first return-time. If it does not exists, we simply say $\tau(x,y)=+\8$. As the first return-time is constant on verticals, we shall also write $\tau(x)$ if there is no ambiguity. 

For $(x,y)\in \CM_{0}$ and in the interior of $\CB$  and $n$ its first return-time, we denote by $\CB(x,y,n)$ the vertical band associated to $(x,y)$ and time $n$. 

The first return map $(x,y)\mapsto F^{\tau(x,y)}(x,y)$ is denoted by $\Phi$. 

\begin{remark}
\label{rem-bandas}
Two different vertical bands $\CB(x,y,n)$ and $\CB(x',y',n')$ have empty interior intersection; more precisely, they can coincide only on one single vertical (respectively  border from the left hand side and from the right hand side).
$\blacksquare$\end{remark}

The set of points in $\CM_{0}$ with finite first return-time may be empty. It is however possible to ensure this set is non-empty. 
Pick some periodic point $P=(x,y)$. It satisfies the past-stabilization property, thus has a piece of local unstable manifold $W^{u}_{lot}(P)$ with length $\delta(P)$. Then, adjust the length $\wh\delta<<\delta(P)$ in the construction of the mille-feuilles such that two consecutive periodic points of the $f$-orbits define an interval which contains $\pi_{1}(P)$. The assumption $\wh\delta<<\delta(P)$ shows that $P$ belongs to the mille-feuilles. As $P$ is periodic, it returns into it by iterations of $F$. 
For the rest of this section, we assume that the set of points with finite first return-time into the mille-feuilles is non-empty.

The different  first-returns generate countably many vertical bands, say $V_{0},V_{1},\ldots$ with disjoint interiors. This defines a ``partition'' called $\CV$:
each band $V_{i}$ is associated to an integer $n_{i}$, and $F^{n_{i}}(V_{i})$ is a horizontal stripe in $\CB$. We denote it by $G_{i}$. By definition of the first return-time, these stripes $G_{i}$ have disjoint interior in $\CB$. We denote by $\CG$ the collection of stripes $G_{i}$. 
Moreover Lemma \ref{lem-nobeaks} yields 
$$F^{n_{i}}(V_{i}\cap \CM_{0})=G_{i}\cap \CM_{0}.$$ 
We can thus consider the induced partitions, $\disp\bigvee_{k=0}^{+\8}\Phi^{-k}\CV$ and $\disp\bigvee_{k=0}^{+\8}\Phi^{-k}\CG$. Formally, 
the dynamics $\Phi$ on $\disp\bigvee_{k=0}^{+\8}\Phi^{-k}\CV\times\disp\bigvee_{k=0}^{+\8}\Phi^{-k}\CG$  is orbit-equivalent to the shift on $\{0,1,\ldots,\}^{\Z}$. 
The one dimensional dynamics $\pi_{1}\circ \Phi$ on $\disp\bigvee_{k=0}^{+\8}\Phi^{-k}\CV$ is orbit-equivalent to  the one-side shift $\{0,1,\ldots,\}^{\N}$. We developp now this later point.

\subsection{One dimensional sub-system}
We define a dynamics $\phi$ on $[P_{l},P_{r}]\cap \disp\bigvee_{k=0}^{+\8}\Phi^{-k}\CV$ by 
$$\phi(x)=\pi_{1}\circ \Phi(x,y),$$
where $y$ is any point such that $(x,y)\in \CM_{0}$. 
Equivalently we have $\phi(x)=f^{n_{i}}(x)$ if $(x,y)$ belongs to $V_{n_{i}}$. 
This can also be written 
$$\phi(x)=f^{\tau(x)}(x).$$
We point out that $\tau(x)$ is not necessarily the first return-time in $[P_{l},P_{r}]$ by iterations of $f$, even if it is the first return-time in $\CM_{0}$ by iterations of $F$ for $(x,y)$. Indeed, it may happen that $F^{k}(x,y)$  belongs to $\CB$ but not to $\CM_{0}$.

The  partition  in vertical band generates a trace on $[P_{l},P_{r}]$: $K_{n_{i}}:=V_{n_{i}}\cap [P_{l},P_{r}]$. This new partition is denoted by $\CK$.

\begin{lemma}
\label{lem-dense-interval}
The set $[P_{l},P_{r}]\cap \disp\bigvee_{k=0}^{+\8}\Phi^{-k}\CV$ is dense in $[P_{l},P_{r}]$.
\end{lemma}
\begin{proof}
Consider a vertical band $\CV_{n_{i}}$ and its image $\CG_{n_{i}}$ by $\Phi$. These two sets intersect themselves and there is a unique $n_{i}$-periodic point in this intersection (with convention $\Phi_{\CV_{n_{i}}}=F^{n_{i}})$).

This proves that there exists periodic points in $\CM_{0}$. 

Hence, consider some $p$-periodic point, say $P=(x_{P},y_{P})$, in $\CM_{0}$. The set $\disp\cup_{j}f^{-j}(\{x_{P}\})$ is dense in $[-1,1]$. Consider $j$  and $x$ in $]P_{l},P_{r}[$ such that $f^{j}(x)=x_{P}$. Pick any $Q$ in $\Lambda$ satisfying $\pi_{1}(Q)=x$. Hence, $F^{j}(Q)$ is in the same vertical line than $P$ (see Figure \ref{Fig-dense}).

Now, consider for some integer $n$ (supposed very big) the connected component of  $\Phi^{-np}(\CB)$ which contains $P$. This is a vertical band, say $V$. 
This set also contains $F^{-j}(Q)$. Then,consider the connected component of $F^{-j}(V)$ which contains $Q$. This is a vertical band, say $V'$. We adjust the integer $n$ such that $V'\subset \inte\CB$. This is possible because $x$ belongs to $]P_{l},P_{r}[$. We also assume that $n$ is sufficiently big such that $f^{k}(V'\cap [P_{l},P_{r}])$ never contains $0$ for $0\le k\le j$. In other words, this interval is never cut (by iteration of $f$) and its image by $f$ is the whole interval $V\cap[P_{l},P_{r}]$.  
Finally  consider any $M$ in $V'\cap \CM_{0}$. The unstable local leaf $W^{u}_{loc}(M)$ overlaps $V'$ in both sides (because it overlaps $\CB$). Our assumptions on $n$ show that $F^{j}(W^{u}_{loc}(M)\cap V')$ is a piece of unstable leaf which overlaps in both sides $V$. Moreover $F^{j}(M)$ is in the band $V$. Say $\Phi^{np}(P)=:F^{k_{n}}(P)$; then the image by $F^{k_{n}}$ of $F^{j}(W^{u}_{loc}(M)\cap V')$ is a piece of unstable manifold which overlaps $\CB$ and $F^{j+k_{n}}(M)$ is in the band $\CB$. Consequently $j+k_{n}$ is a return-time for $M$ into $\CM_{0}$.

\begin{figure}[htbp]
\begin{center}
\includegraphics[scale=0.5]{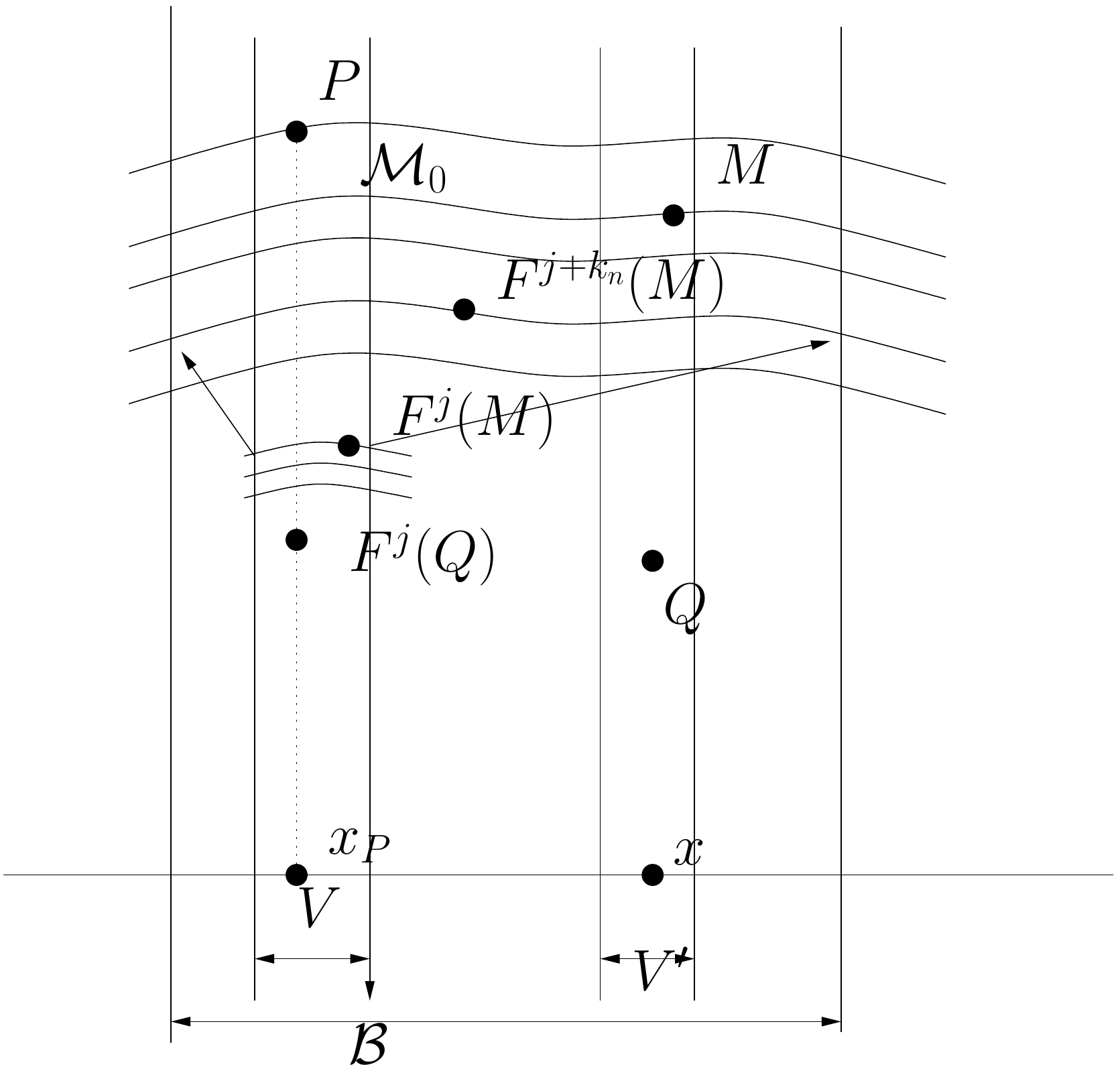}
\caption{density of $\CK$}
\label{Fig-dense}
\end{center}
\end{figure}

\end{proof}

We fix some unstable leaf $\CF_{0}$ of the mille-feuilles and consider it as the reference leaf. The projection $\pi_{\CF_{0}}$ on $\CF_{0}$ is defined by 
$$\pi_{\CF_{0}}(x,y)=(x,y_{0}),$$
where $y_{0}$ is such that $(x,y_{0})\in \CF_{0}$. 
To simplify the notations,and if it does not make difficulties to be understood, we shall also write $\pi_{\CF_{0}}(x)=y_{0}$, to mean that $(x,\pi_{\CF_{0}}(x))$ belongs to $\CF_{0}$. 

\bigskip
If $x$ and $x'$ are in $[P_{l},P_{r}]$, we set 
$$d(x,x')=\frac1{2^{\max\{n,\ |f^{n}(x)-f^{n}(x')|\le \wh\delta\}}},$$
where $\wh\delta$ is  defined in \eqref{def-delta-band} and is bigger than the size of the band $\CB$. 
As $f$ is expanding, if $x\neq x'$, eventually, $|f^{n}(x)-f^{n}(x')|>\wh\delta$ if the interval $[x,x']$ is not cut by 0. If it is cut by 0 (with length smaller than $\wh\delta$) at some iterate, say $n_{0}$, condition \eqref{equ2-condi-deltabanda} shows that one point is sent to $[-1,-\wh\delta[$ and the other one to $]\wh\delta,1]$. At that moment $n_{0}+1$, $|f^{n}(x)-f^{n}(x')|>\wh\delta$. In both cases, ${\max\{n,\ |f^{n}(x)-f^{n}(x')|\le \wh\delta\}}<+\8$. 

As $f$ is expanding by a factor at least $\sqrt2$,  for every $x$ and $x'$
\begin{equation}
\label{equ-majo-ditance}
|x-x'|\le 2^{\frac{\log\wh\delta}{\log2}+\frac12}\sqrt{d(x,x')}.
\end{equation}

\begin{definition}
\label{def-Cgamma}
We define the set $\CC^{\gamma}$ of continuous functions  $\varphi:[P_{l},P_{r}]\to\R$ satisfying
$$\sup_{x\neq x'}\frac{|\varphi(x)-\varphi(x')|}{d^{\gamma}(x,x')}<+\8.$$
 For $\varphi$ in $\CC^{\gamma}$, we set 
 $$||\varphi||_{\gamma}:=\sup_{x\neq x'}\frac{|\varphi(x)-\varphi(x')|}{d^{\gamma}(x,x')}+||\varphi||_{\8}.$$
\end{definition}
Clearly a function in $\CC^{\gamma}$ is continuous because $d(x,x')$ goes to 0 if $x$ goes to $x'$. Then  $||\varphi||_{\8}$ and $||\varphi||_{\gamma}$ are well defined. 

We let the reader check the next result:

\begin{proposition}
\label{prop-ccgammatoubon}
$||\ ||_{\gamma}$ is a norm. The normed space $(\CC^{\gamma},||\ ||_{\gamma})$ is a Banach space. 
\end{proposition}

\section{Proof of Theorem D}\label{sec-thD}
\paragraph{Notation.} In this section we will consider several dynamical systems: $(\Lambda, F)$, $(\CM_{0},\Phi)$, and $([P_{l},P_{r}],\phi)$. Then, for each Birkhoff sum we shall write with respect to which dynamics it is considered. Namely we will write $S_{n}^{F}$, $S_{n}^{\Phi}$ and $S_{n}^{\phi}$. 

\subsection{Local equilibrium state for induced map}\label{subsec-reduc1dim}
\subsubsection{Reduction to a one-dimensional dynamics}
Let $A_{0}$ be a H\"older continuous function on $[-1,1]^{2}$. We set 
$$\omega(x,y)=\sum_{k=0}^{+\8}A_{0}\circ F^{k}(x,y)-A_{0}\circ F^{k}\circ \pi_{\CF_{0}}(x,y).$$
Assume $\tau(x,y)=n$ and set $\pi_{\CF_{0}}(x)=y_{0}$. Then we have 
\begin{eqnarray*}
\omega(x,y)&=&\sum_{k=0}^{+\8}A_{0}\circ F^{k}(x,y)-A_{0}\circ F^{k}\circ \pi_{\CF_{0}}(x,y)\\
&=& S_{n}^{F}(A_{0})(x,y)-S_{n}^{F}(A_{0})(x,y_{0})+\sum_{k=0}^{+\8}A_{0}\circ F^{k}(F^{n}(x,y))-A_{0}\circ F^{k}(F^{n}(x,y_{0}))\\
&=& S_{n}^{F}(A_{0})(x,y)-S_{n}^{F}(A_{0})(x,y_{0})+\omega(\Phi(x,y))-\sum_{k=0}^{+\8}A_{0}\circ F^{k}\circ\pi_{\CF_{0}}\circ \Phi(x,y)-A_{0}\circ F^{k}(F^{n}(x,y_{0})).
\end{eqnarray*}
 This yields
 $$S_{n}^{F}(A_{0})(x,y)=S_{n}^{F}(A_{0})(x,y_{0})-\sum_{k=0}^{+\8}\left(A_{0}\circ F^{k}\circ\pi_{\CF_{0}}\circ \Phi(x,y)-A_{0}\circ F^{k}(F^{n}(x,y_{0}))\right)+\omega(x,y)-\omega\circ\Phi(x,y).$$
Note that $\disp A(x):=S_{n}(A_{0})(x,y_{0})-\sum_{k=0}^{+\8}\left(A_{0}\circ F^{k}\circ\pi_{\CF_{0}}\circ \Phi(x,y)-A_{0}\circ F^{k}(F^{n}(x,y_{0}))\right)$ does not depend on $y$. We have 
\begin{equation}
\label{equ1-cobord}
S_{n}^{F}(A_{0})(x,y)=A(x)+\omega\circ \Phi(x,y)-\omega(x,y).
\end{equation}
Let $Z$ be a real number. The last equality shows that 
it is equivalent to find an equilibrium state for $(x,y)\mapsto S_{\tau(x,y}^{F}(A_{0})(x,y)-\tau(x,y)Z$ (and for the dynamical system $(\disp\bigvee_{k=0}^{+\8}\Phi^{-k}\CV\times\disp\bigvee_{k=0}^{+\8}\Phi^{-k}\CG,\Phi)$) or for $(x,y)\mapsto A(x)-\tau(x).Z$. 

\subsubsection{Study of the one-dimensional dynamics}\label{subsubsec-thermorenaud}

\begin{lemma}
\label{lem-bongamma}
There exists a positive $\gamma$  and a constant $C_{A}$ such that for every $x$ and $x'$ in the same $K_{n}$, 
$$\left|A(x)-A(x')\right|\le C_{A}d^{\gamma}(\phi(x),\phi(x')).$$ 
\end{lemma}
\begin{proof}
We pick $x$ and $x'$ in the same $K_{n}$. By definition we have $\tau(x)=\tau(x')=n$. Then 
$$\max\{k,\ |f^{k}(x)-f^{k}(x')|\le \delta\}=n+m,$$
for some non-negative integer $m$. We have $d(\phi(x),\phi(x'))=2^{-m}$.

We have to compute the difference 
$$\sum_{k=0}^{+\8}A_{0}\circ F^{k}\circ\pi_{\CF_{0}}\circ\Phi(x,\pi_{\CF_{0}}(x))-A_{0}\circ F^{k}\circ \Phi(x,\pi_{\CF_{0}}(x))-\sum_{k=0}^{+\8}A_{0}\circ F^{k}\circ\pi_{\CF_{0}}\circ\Phi(x',\pi_{\CF_{0}}(x'))-A_{0}\circ F^{k}\circ \Phi(x',\pi_{\CF_{0}}(x')).$$
It is well-known that this is done by cutting the sum in two parts.
 
On the one hand, we compute bound for 
$$\sum_{k=0}^{\frac{m}2}A_{0}\circ F^{k}\circ\pi_{\CF_{0}}\circ\Phi(x,\pi_{\CF_{0}}(x))-A_{0}\circ F^{k}\circ\pi_{\CF_{0}}\circ\Phi(x',\pi_{\CF_{0}}(x'))$$
and 
$$\sum_{k=0}^{\frac{m}2}A_{0}\circ F^{k}\circ \Phi(x,\pi_{\CF_{0}}(x))-A_{0}\circ F^{k}\circ \Phi(x',\pi_{\CF_{0}}(x')).$$
On the other hand we compute bound for 
$$\sum_{k=\frac{m}2+1}^{+\8}A_{0}\circ F^{k}\circ\pi_{\CF_{0}}\circ\Phi(x,\pi_{\CF_{0}}(x))-A_{0}\circ F^{k}\circ \Phi(x,\pi_{\CF_{0}}(x))$$
and 
$$\sum_{k=\frac{m}2+1}^{+\8}A_{0}\circ F^{k}\circ\pi_{\CF_{0}}\circ\Phi(x',\pi_{\CF_{0}}(x'))-A_{0}\circ F^{k}\circ \Phi(x',\pi_{\CF_{0}}(x')).$$
For both primary terms, we use that $f^{k}(x)$ and $f^{k}(x')$ are close. For both secondary terms we use that $F^{k}\circ\pi_{\CF_{0}}\circ\Phi(x,\pi_{\CF_{0}}(x))$ is close to $F^{k}\circ \Phi(x,\pi_{\CF_{0}}(x))$ and $F^{k}\circ\pi_{\CF_{0}}\circ\Phi(x',\pi_{\CF_{0}}(x'))$ is close to $F^{k}\circ \Phi(x',\pi_{\CF_{0}}(x'))$. We then use the H\"older regularity of $A_{0}$. The unstable leave is a Lipschitz graph, hence $\pi_{\CF_{0}}$ is Lipschitz continuous. 

All this allows to get a bound of the form $C_{A}\theta^{m}$, for some $0<\theta<1$  and $C_{A}$ which only depends on $A_{0}$.  The quantity $\theta^{m}$ is equal to 
 $d^{\gamma}(\phi(x),\phi(x'))$ with $\gamma=\disp\frac{|\log\theta|}{\log2}$. 
\end{proof}

\begin{definition}
\label{def-clz}
We define the transfer operator with parameter $Z$ by 
$$\CL_{Z}(\psi)(x)=\sum_{\phi(\xi)=x}e^{A(\xi)-\tau(\xi)Z}\psi(\xi).$$
\end{definition}

\begin{proposition}
\label{prop-zclz}
There exists a {\em critical} $Z_{c}$ such that for every $Z>Z_{c}$, for every continuous $\psi=[P_{l},P_{r}]\to \R$, for every $x$ in $[P_{l},P_{r}]$, 
the quantity $\CL_{Z}(\psi)(x)$ is well defined. 

$Z_{c}$ is critical in the sense that it is the minimal value with this property. 

For $Z>Z_{c}$, $\CL_{Z}$ is a linear operator acting on the set $\CC^{0}$ of continuous function on $[P_{l},P_{r}]$. 
\end{proposition}
\begin{proof}
The proof can be found in \cite{leplaideur1} Subsec. 4.1. 

Using Lemma \ref{lem-bongamma}, we show that convergence for  $Z$, for every $\psi$ and every $x$ is equivalent to convergence for $Z$, $\BBone$ and just one $x$. Then we get 
\begin{equation}
\label{eq-defzc}
Z_{c}=\limsup_{\ninf}\frac1n\log\left(\sum_{\phi(\xi)=x,\ \tau(\xi)=n}e^{A(\xi)}\right).
\end{equation}
Then, $\CL_{Z}$ acts on continuous functions because even if $\phi$ is not defined for every $x$, the inverse branches are well-defined and the operator is Markov: for any $x$ and $x'$ we can associate by pair the pre-images $\xi$ and $\xi'$. Moreover, there is contractions iterating backward. 
\end{proof}

\begin{proposition}
\label{prop-distoclzn}
There exists a positive constant $C_{A}$ which only depends on $A_{0}$ such that for every $Z>Z_{c}$, for every $x$ and $x'$ in $[P_{l},P_{r}]$, 
$$e^{-C_{A}}\le\frac{\CL_{Z}^{n}(\BBone)(x)}{\CL_{Z}^{n}(\BBone)(x')} \le e^{C_{A}}.$$
\end{proposition}
\begin{proof}
This is a direct consequence of Lemma \ref{lem-bongamma}. If $\xi$ satisfies $\phi^{n}(\xi)=x$, then there exists $\xi'$ in the same element $\disp\bigvee_{k=0}^{n}\phi^{-k}(\CK)$ satisfying $\phi^{n}(\xi')=x'$. 
Lemma \ref{lem-bongamma} yields
$$\left|S_{n}^{\phi}(A)(\xi)-S_{n}^{\phi}(A)(\xi')\right|\le C_{A},$$
where $S_{n}^{\phi}(A)=A+A\circ\phi+\ldots+A\circ \phi^{n-1}$ and $C_{A}$ is a constant only depending on $A_{0}$. 
\end{proof}

\paragraph{\bf Results following \cite{leplaideur1}.}
We can use a theorem from Ionescu-Tulcea \& Marinescu \cite{ionescu-tulcea-marinescu} and we get for every $Z>Z_{c}$:

$\bullet$ There is an eigen- probability measure $\nu_{Z}$ for $\CL_{Z}^{*}$. The eigenvalue $\lambda_{Z}$ is the spectral radius of $\CL_{Z}$ and $\CL_{Z}^{*}$. 

$\bullet$ On $\CC^{\gamma}$, the spectrum of $\CL_{Z}$ is a single and simple dominated eigenvalue $\lambda_{Z}$, and the rest of the spectrum included into a disk of radius $\rho_{Z}\lambda_{Z}$ with $\rho_{Z}<1$. 

$\bullet$ The H\"older inequalities shows that 
$Z\mapsto\log\lambda_{Z}$ is convex. It is also decreasing and analytic on $]Z_{c}+\8[$.

$\bullet$ The unique (up to a multiplicative constant) eigen-function is 
$$H_{Z}:=\lim_{\ninf}\frac1n\sum_{k=0}^{n-1}\frac{\CL_{Z}^{k}(\BBone)}{\lambda_{Z}^{k}}.$$
It is a positive function in $\CC^{\gamma}$. 

$\bullet$  The measure defined by $d\mu_{Z}:=H_{Z}\,d\nu_{Z}$ is the unique equilibrium state for $([P_{l},P_{r}],\phi)$ associated to $A-Z.\tau$. It is an exact measure, hence mixing, hence ergodic (mixing yields uniqueness of the dominated eigenvalue for $\CL_{Z}$). 

$\bullet$ Uniqueness of the dominated eigenvalue and \cite{Hennion-Herve} chap. 3 yield that $Z\mapsto\lambda_{Z}$ is real analytic on $]Z_{c},+\8[$. 

$\bullet$ For every $x$, the set of pre-images by $\phi$ is dense in $\bigcup_{K\in \CK} K$. Moreover, Lemma \ref{lem-dense-interval} shows that $\bigcup_{K\in \CK} K$ is dense in $[P_{l},P_{r}]$.   

$\bullet$ All these results are valid as soon as $\CL_{Z}(\BBone)$ converges. This holds for $Z=Z_{c}$ if $\CL_{Z_{c}}(\BBone)$ converges (see \cite{leplaideur1} subset. 6.3).

\subsubsection{Extensions to the two-dimensional dynamics}
Copying \cite{leplaideur1}, we claim that $(\CM_{0}, \Phi)$ is the natural extension for $([P_{l},P_{r}],\phi)$. Indeed, and by construction of $\CM_{0}$, the discontinuity  generated by the line $\{x=0\}$ is not seen by $\phi$ and the Markov partition $\CK$. Therefore, for every $Z$, there exists an unique $\Phi$-invariant measure $\wh\mu_{Z}$ such that $\pi_{1*}\wh\mu_{Z}=\mu_{Z}$. 
It satisfies  $$h_{\wh\mu_{Z}}(\Phi)+\int A-Z\tau\,d\wh\mu_{Z}=\log\lambda_{Z}.$$
and it is the unique equilibrium state for $(\CM_{0},\Phi)$ and the potential $A-Z.\tau$. 
 
 By Equality \eqref{equ1-cobord}, $\wh\mu_{Z}$ is also an equilibrium state for $S_{\tau}^{F}(A_{0})-Z.\tau$.

 Moreover, for every $Z>Z_{c}$, we recall that $F$ lets verticals invariant (and contracts them). Moreover the density function $H_{Z}$ is bounded from below away from 0. Therefore, 
 the condition $\int \tau\,d\wh\mu_{Z}<+\8$ is equivalent to $\int\tau\,d\mu_{Z}<+\8$ and is also equivalent to $\int\tau\,d\nu_{Z}<+\8$. 
 
Now,  $\int\tau\,d\nu_{Z}<+\8$ holds if $Z>Z_{c}$ because Lemma \ref{lem-bongamma}
 shows 
 $$\int\tau\,d\nu_{Z}=-e^{\pm C_{A}}\frac{\partial \CL_{Z}(\BBone)(x)}{\partial Z},$$
 for any $x$. 
 
 Then, following \cite{dowker}, there exists $m_{Z}$, a  $F$-invariant probability, such that 
$$\wh\mu_{Z}=\frac{m_{Z}(.\cap \CM_{0})}{m_{Z}(\CM_{0})}.$$

In that case we get 
\begin{equation}
\label{equ-pressZgrand}
h_{m_{Z}}(F)+\int A_{0}\,dm_{Z}=Z+m_{Z}(\CM_{0})\log\lambda_{Z}.
\end{equation}

Moreover \begin{equation}
\label{equ-derivloglambda}
\frac{d\log\lambda_{Z}}{dZ}=\frac{-1}{m_{Z}(\CM_{0})}
\end{equation} 

Again, these results hold, if $\CL_{Z}(\BBone)$ converges for $Z=Z_{c}$ and the existence of $m_{Z}$ from the convergence of $\frac{\partial \CL_{Z}(\BBone)(x)}{\partial Z}$ for $Z=Z_{c}$ (and $x$ is any point). 

\subsection{Relative equilibrium associated to $\CM_{0}$}
\begin{definition}
\label{def-P0m0}
We call relative pressure for $A_{0}$ associated to $\CM_{0}$ the quantity:
$$\CP(A_{0},\CM_{0}):=\sup\left\{h_{m}(F)+ \int A_{0}\,dm,\ m(\CM_{0})>0\right\}.$$
Any $F$-invariant measure giving positive weight to $\CM_{0}$ and realizing this supremum is called a relative equilibrium state associated to $\CM_{0}$. 
\end{definition}
The goal of this subsection is to prove the next proposition and to make precise condition yielding existence (and uniqueness) of a relative equilibrium state associated to $\CM_{0}$. 

\begin{proposition}
\label{prop-maxmZ}
$\CP(A_{0},\CM_{0}):=\sup\left\{h_{m_{Z}}(F)+ \int A_{0}\,dm_{Z},\ Z)>Z_{c}\right\}.$
\end{proposition}

\subsubsection{Key estimation for $Z_{c}$}

 Here, follows the key estimation for the proof of Theorem D.

\begin{proposition}
\label{prop-estizc}
Inequality $\disp Z_{c}\le \CP(A_{0},\CM_{0})$ holds. 
\end{proposition}
\begin{proof}
Let $n$ be an integer, and consider the finite vertical bands $\CV_{i}$ associated to the return time $\tau=n$ in the mille-feuilles $\CM_{0}$. 
The Markov intersection yields in each such band the existence and uniqueness of a $\phi$ periodic  point with period (exactly) equal to $n$. 

Denote by $M_{n,j}$ such periodic point and $K_{n,j}$ the associated cylinder. We set $A_{n,j}:=A(M_{n,j})$ and consider the Bernoulli measure\footnote{for the dynamics of $\phi$.} $\mu_{n}$ with weight 
$$p_{n,j}:=\frac{e^{A_{n,j}}}{\sum_{j}e^{A_{n,j}}}.$$
Bernoulli means here $\disp\mu_{n}(\cap_{j=0}^{p}\phi^{-j}\CV_{i_{j}})=\prod_{j=0}^{p}p_{n,i_{j}}$.

Therefore, $\disp h_{\mu_{n}}(\phi)=-\sum p_{n,j}\log p_{n,j}=-\sum_{j}\frac{A_{n,j}e^{A_{n,j}}}{\sum_{i}e^{A_{n,i}}}+\log\sum_{j}e^{A_{n,j}}$. 
Lemma \ref{lem-bongamma}, shows 
$\disp\int A\,d\mu_{n}=\pm C_{A}+\sum_{j}A_{n,j}p_{n,j}$,
and then 
$$h_{\mu_{n}}(\phi)+\int A\,d\mu_{n}=\log\left(\sum_{j}e^{A_{n,j}}\right)\pm C_{A}.$$
Note that $\disp\int \tau\,d\mu_{n}=n$, hence there exists a $F$-invariant probability $m_{n}$ such that 
$$\mu={\pi_{1}}_{*}\left(\frac{m_{n}(.\cap \CM_{0})}{m_{n}(\CM_{0})}\right).$$
Furthermore $m_{n}(\CM_{0})=\disp\frac1n$. Consequently the free energy of $A_{0}$\footnote{Remember that $A$ is up to a coboundary for the first return into $\CM_{0}$ the induced potential $S_{\tau(.)}(A_{0})(.)$} for $m_{n}$ satisfies
$$\CP(A_{0},\CM_{0})\ge h_{m_{n}}(F)+\int A_{0}\,dm_{n}=\frac1n\log\left(\sum_{j}e^{A_{n,j}}\right)\pm \frac1nC_{A}.$$
This yields: 
\begin{equation}
\label{equ-majozc}
\frac1n\log\left(\sum_{j}e^{A_{n,j}}\right)\le \CP(A_{0},\CM_{0})+\frac1nC_{A}.
\end{equation}

On the other hand, let $x$ be any point in $[P_{l},P_{r}]$. There exists a unique preimage $\xi_{j}$ of $x$ in $\CK_{i}$ (for $\phi$). We recall Equality \eqref{eq-defzc} with these notations: 
$$Z_{c}=\limsup_{n\to+\8}\frac1n\log\left(\sum_{j}e^{A(\xi_{j})}\right).$$
Lemma \ref{lem-bongamma} shows that each $A(\xi_{j})$ can be replaced by $A(M_{n,j})=A_{n,j}$.

Hence, considering a subsequence of $n$'s such that $\disp\frac1n\log\left(\sum_{j}e^{A_{n,j}}\right)$ converges to $Z_{c}$ and doing $n\to +\8$, Inequality \eqref{equ-majozc} shows that  $Z_{c}\le \CP(A_{0},\CM_{0})$ holds. 
\end{proof}

\subsubsection{Existence and uniqueness for relative equilibrium}\label{subsubsec-relatequilmille}
Here we prove Proposition \ref{prop-maxmZ} and state condition yielding existence or not of a relative equilibrium state for $A_{0}$ associated to the mille-feuilles $\CM_{0}$.

Let $A_{0}$  and $\CM_{0}$ be as above.
Set 
\begin{equation}
\label{eq-def-cpbeta}
\CP(\be,\CM_{0}):=\sup\left\{h_{m}(F)+\int A_{0}\,dm+\be.m(\CM_{0}),\ m(\CM_{0})>0\right\}.
\end{equation}

\newcommand{\clo}{\CL_{Z}}
\newcommand{\bbone}{\BBone_{[P_{l},P_{r}]}}

We show  a relation between $\be$ and $Z$. 
\begin{lemma}
\label{lem-relbetaZ}
For $Z>Z_{c}$, $m_{Z}$ realizes the maximum in Equality \eqref{eq-def-cpbeta} with $\be:=-\log\lambda_{Z}$. This maximal value is equal to $Z$. 
\end{lemma}
\begin{proof}
Pick $Z>Z_{c}$.
Set $\be:=-\log\lambda_{Z}$ and consider some measure $m$, $F$-invariant such that $m(\CM_{0})>0$. We denote by $\wh\mu$ the conditional measure 
$$\wh\mu:=\frac{m(.\cap\CM_{0})}{m(\CM_{0})}.$$
It is a $\Phi$-invariant probability. 

We recall that $\wh\mu_{Z}$ is the local equilibrium state for $(\CM_{0},\Phi)$ and the potential $S_{\tau(.)}^{F}(A_{0})-Z\tau(.)$. This yields
\begin{eqnarray*}
h_{\wh\mu}(\Phi)+\int S_{\tau(\xi)}^{F}(A_{0})(\xi)-Z.\tau(\xi)\,d\wh\mu&\le &\log\lambda_{Z}=-\be\\
&&\text{ with equality iff }\wh\mu=\wh\mu_{Z}\\
&\Updownarrow&\\
 h_{m}(F)+\int A_{0}\,dm+\be.m(\CM_{0})&\le & Z\\
&&\text{ with equality iff }m=m_{Z}.
\end{eqnarray*}
\end{proof}

The main consequence of Lemma \ref{lem-relbetaZ} is the formula 
\begin{equation}
\label{equ-egalcpbe-Z}
\CP(-\log\lambda_{Z},\CM_{0})=Z.
\end{equation}
We let the reader check that $\be\mapsto \CP(\be,\CM_{0})$ is increasing and convex, thus continuous. Moreover, $Z\mapsto \log\lambda_{Z}$ is decreasing. Therefore $\lim_{Z\downarrow Z_{c}}\log\lambda_{Z}$ exists. Let us denote it by $-\be_{c}\le +\8$. Hence, Equality \eqref{equ-egalcpbe-Z} also holds for $Z=Z_{c}$ and $-\be_{c}$ instead of $\log\lambda_{Z}$. 
We recall that by Proposition \ref{prop-estizc}, $Z_{c}\le \CP(A_{0},\CM_{0})=\CP(0,\CM_{0})$. Hence we have 
 $$\CP(\be_{c},\CM_{0})=Z_{c}\le \CP(0,\CM_{0}),$$
 which implies $\be_{c}\le 0$. In other words, 
 \begin{equation}
\label{equ-limloglambdaZ}
\lim_{Z\downarrow Z_{c}}\log\lambda_{Z}\ge 0.
\end{equation}

We let the reader check the next technical result. 
\begin{lemma}
\label{lem-convexite}
Let $\varphi$ be a convex and decreasing continuously differentiable function defined on $]a,b[$. Then $\psi:z\mapsto z-\disp\frac{\varphi(z)}{\varphi'(z)}$ attains its maximal value, either on  the unique $z_{0}$ such that $\varphi(z_{0})=0$ or on $a$ (by continuity) if $\varphi<0$ on $]a,b[$. 
\end{lemma}

Equality \eqref{equ-derivloglambda} can be re-written 
\begin{equation}
\label{equ-mzmax}
h_{m_{Z}}(F)+\int A_{0}\,dm_{Z}=Z-\frac{\log\lambda_{Z}}{\frac{d\log\lambda_{Z}}{dZ}},
\end{equation}
setting $\varphi(Z):=\log\lambda_{Z}$. As $\lim_{Z\to+\8}\varphi(Z)=-\8$, Equality \eqref{equ-limloglambdaZ} shows that either there exists a unique $Z>Z_{c}$ such that $\lambda_{Z}=1$ (\ie  $\ \varphi(Z)=0$), or $\lim_{Z\downarrow Z_{c}}\lambda_{Z}=1^{-}$. 

Note that  the function $\psi$ in Lemma \ref{equ-mzmax} is continuous, thus 
$$\CP(A_{0},\CM_{0})=\CP(0,\CM_{0})=\sup\{h_{m_{Z}}(F)+\int A_{0}\,dm_{Z},\ \lambda_{Z}<1\}.$$
This finishes the proof of Proposition \ref{prop-maxmZ}.


%

%
%
%
%
%

\bigskip
Now, we state conditions yielding existence uniqueness and/or non-existence of a relative equilibrium state associated to $\CM_{0}$. 

\begin{theorem}
\label{theo-equil-CM0}
 Then, there exists  a relative equilibrium state for $A_{0}$ associated to $\CM_{0}$ if and only if the right derivative  for $\be\mapsto \CP(\be,\CM_{0})$ at $0$ is positive. 
\end{theorem}
\begin{proof}
The function $\be\mapsto \CP(\be,\CM_{0})$ is non-decreasing and convex. It thus admits a left and right derivative at any point. They are non-negative.

Let us first assume that there exists some relative equilibrium state for $A_{0}$ associated to $\CM_{0}$, say $m$. 
Then, for every $\be>0$, $\CP(\be,\CM_{0})\ge \CP(A_{0},\CM_{0})+\be.m(\CM_{0})$. This yields that the right derivative for $\CP(\be,\CM_{0})$ at $0$ is at least $m(\CM_{0})>0$.

Let us prove the converse. Assume that the left derivative at $0$ is $\al>0$.
For every $\eps>0$ and $\be>0$, we can find $m_{\eps,\be}$ such that $m_{\eps,\be}(\CM_{0})>0$ and 
$$h_{m_{\eps,\be}}(F)+\int A_{0}\,dm_{\eps,\be}+\be.m_{\eps,\be}(\CM_{0})>\CP(\be,\CM_{0})-\eps.$$
The assumption on the left derivative and the convex property yield
$$\CP(\be,\CM_{0})\ge \CP(0,\CM_{0})+\be\al.$$
Hence, we can choose $\eps=\be^{2}$ and assume that $\be$ is sufficiently small such that 
\begin{equation}
\label{equ1-minopression}
h_{m_{\eps,\be}}(F)+\int A_{0}\,dm_{\eps,\be}+\be.m_{\eps,\be}(\CM_{0})>\CP(0,\CM_{0})+\be\frac\al2
\end{equation}
holds.  Now, by definition $\disp h_{m_{\eps,\be}}(F)+\int A_{0}\,dm_{\eps,\be}\le \CP(0,\CM_{0})=\CP(A_{0},\CM_{0})$. Thus we get for every $\be>0$
$$\be m_{\be^{2},\be}(\CM_{0})\ge \be\frac\al2.$$
Therefore, for every sufficiently small positive $\be$, $m_{\be^{2},\be}(\CM_{0})>\frac\al2$. 

\medskip
Let us consider any accumulation point $m$ for $m_{\be^{2},\be}$ as $\be\to0$. We claim that $m(\CM_{0})$ is positive. Indeed, either $m(\partial\CM_{0})=0$ and then a standard computation yields
$$m(\CM_{0})\ge \frac\al2,$$
or $m(\partial\CM_{0})>0$ and then $m(\CM_{0})>0$ because $\CM_{0}$ is compact, thus $\partial\CM_{0}\subset\CM_{0}$.

We have seen above (see Corollary \ref{cor-entropy-scs}) that the entropy is upper semi-continuous. Therefore,  doing $\be\to0$, \eqref{equ1-minopression} implies
$$h_{m}(F)+\int A_{0}\,dm\ge \CP(A_{0},\CM_{0}).$$
Hence,  $m$ is a relative equilibrium state for $A_{0}$ associated to $\CM_{0}$. 
\end{proof}

Consequently, and according to Subsubsection \ref{subsubsec-thermorenaud} we get 3 possible cases (see Fig. \ref{fig-3cases}):
\begin{enumerate}
\item $Z_{c}<\CP(0,\CM_{0})$ and $m_{Z}$ with $Z=\CP(0,\CM_{0})$ is the unique relative equilibrium state for $A_{0}$ associated to $\CM_{0}$. In that case $\lambda_{Z}=1$. 

\item $Z_{c}=\CP(0,\CM_{0})$ and $\disp\left|\frac{d\log\lambda_{Z}}{dZ}\right|<+\8$. Then, $m_{Z}$ with $Z=Z_{c}$ is the unique relative equilibrium state for $A_{0}$ associated to $\CM_{0}$. In that case $\lambda_{Z_{c}}=1$. 

\item $Z_{c}=\CP(0,\CM_{0})$ and $\disp\left|\frac{d\log\lambda_{Z}}{dZ}\right|=+\8$. Then, there is no relative equilibrium state for $A_{0}$ associated to $\CM_{0}$. In that case $\lambda_{Z_{c}}=1$. 
\end{enumerate}

\begin{figure}[htbp]
\begin{center}
\includegraphics[scale=0.5]{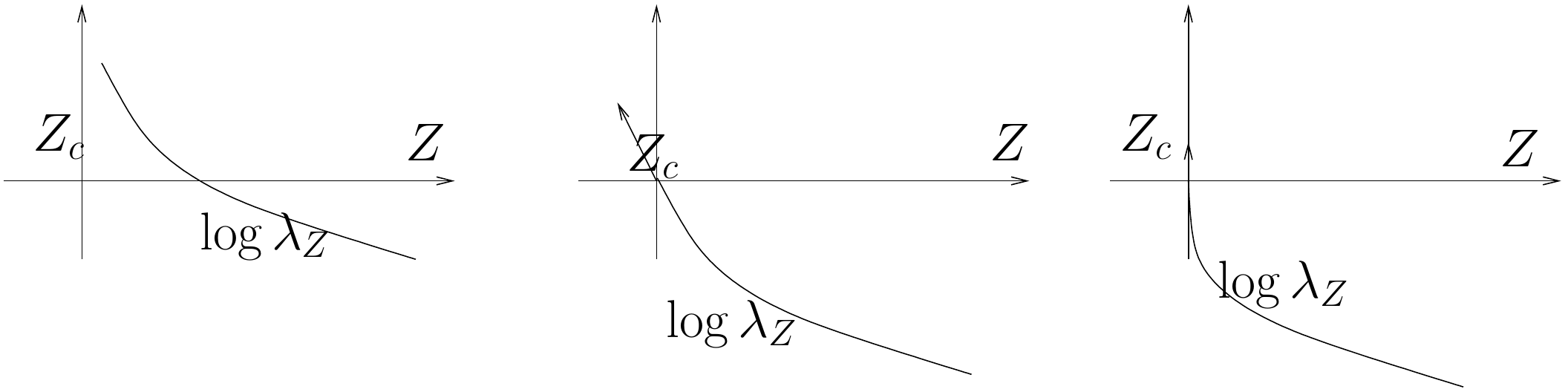}
\caption{3 possibles cases}
\label{fig-3cases}
\end{center}
\end{figure}

\subsection{Existence and uniqueness of the relative equilibrium among good measures}

\begin{definition}
\label{def-fatmeas}
We say that a $F$-invariant measure $m$ is fat if every mille-feuilles has positive weight.
\end{definition}

\begin{proposition}
\label{prop-maxfat}
The supremum of the free energies among good measures is also the supremum  among fat measures:
$$\sup\left\{h_{m}(F)+\int A_{0}\,dm,\ m\text{ is good}\right\}=\sup\left\{h_{m}(F)+\int A_{0}\,dm,\ m\text{ is fat}\right\}.$$
\end{proposition}
\begin{proof}
Actually, we shall prove that any measure of the form $m_{Z}$ is fat. Remember that $m_{Z}$ gives positive weight to any open vertical band in $\CM_{0}$. 

Then, we just copy the proof of Lemma \ref{lem-dense-interval}. We pick another mille-feuilles, say $\CM_{1}$ and consider some periodic point inside $\CM_{1}$ (see Figure \ref{fig-fatmz}). 

\begin{figure}[htbp]
\begin{center}
\includegraphics[scale=0.6]{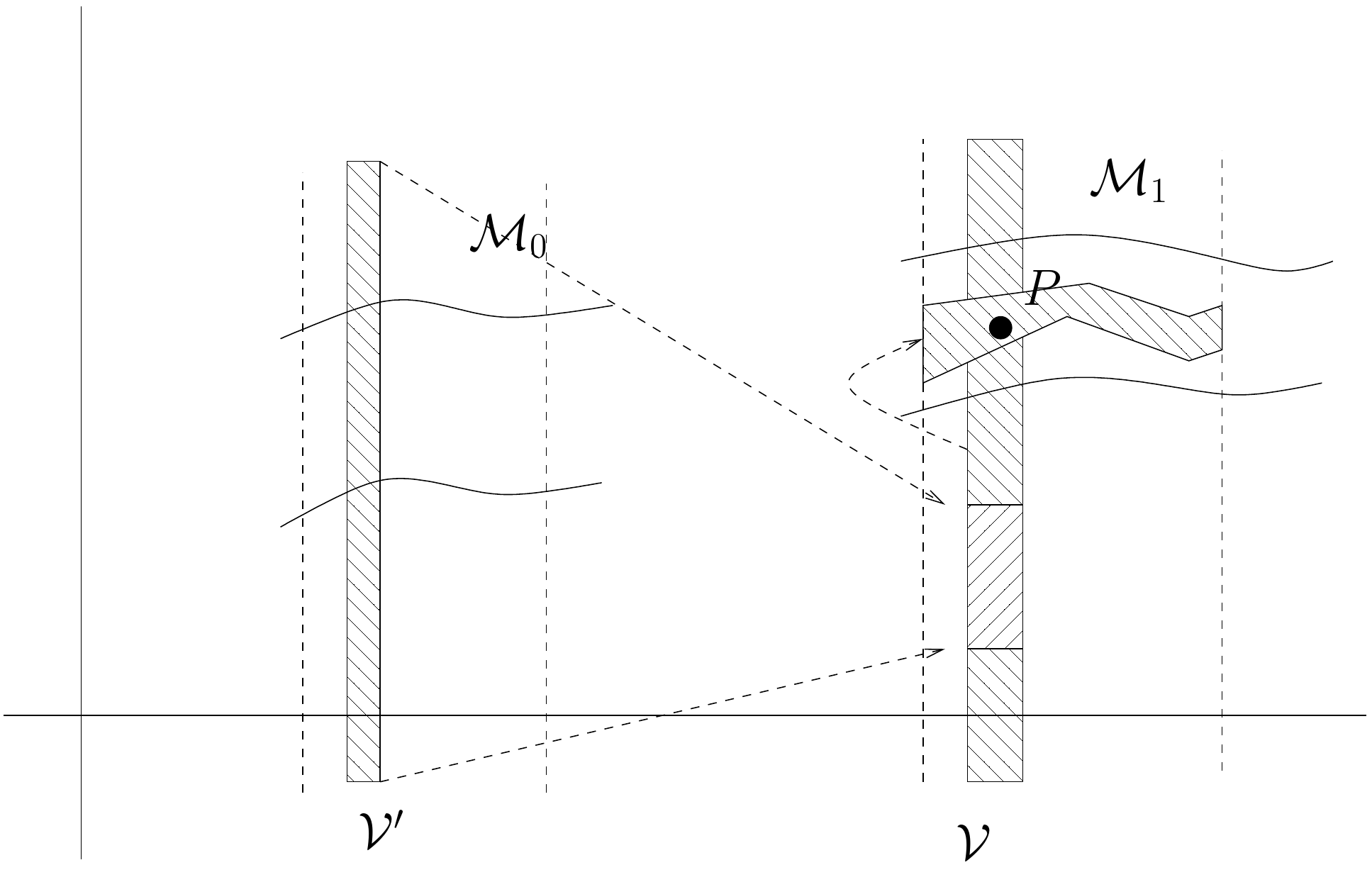}
\caption{$\CM_{1}$ has positive $m_{Z}$-measure}
\label{fig-fatmz}
\end{center}
\end{figure}

Then we consider a very thin vertical band $\CV$ whose size will be specified later. We consider some preimage of the $x$-coordinate of the periodic point which belongs to the interior of the vertical band defining $\CM_{0}$ and such that the vertical band $\CV'$ mapped onto $\CV$ (in the horizontal direction) is also inside the vertical band defining $\CM_{0}$ (here we adjust the size of $\CV$). This vertical band $\CV'$ has positive $\wh\mu_{Z}$-measure, hence positive $m_{Z}$-measure. It is mapped by the good iterate of $F$ into the vertical band $\CV$ and then into some horizontal stripe $\CG$ in $\CM_{1}$. This proves that $\CM_{1}$ has positive $m_{Z}$-measure.

\end{proof}

As an immediate consequence of Proposition \ref{prop-maxfat} we get
\begin{proposition}
\label{prop-samepress}
Let $\CM_{0}$ and $\CM_{1}$ be two mille-feuilles. Then $\CP(A_{0},\CM_{0})=\CP(A_{0},\CM_{1})$.
\end{proposition}
\begin{proof}
Any fat measure gives positive weight to any mille-feuilles. Then, 
$$\CP(A_{0},\CM_{0})=\sup\left\{h_{m}(F)+\int A_{0}\,dm,\ m\text{ is fat}\right\}=\CP(A_{0},\CM_{1}).$$
\end{proof}
Then, the description of the 3 possible cases at the end of Subsubsection \ref{subsubsec-relatequilmille}  concludes the proof of Theorem D. 

\appendix
\section{Relative and global equilibrium state}\label{sec-appen}

Even if we presented here a way to be somehow emancipated from the question of global equilibrium state, the question still remains. In particular, when is the relative equilibrium state among good measure also a/the global equilibrium ?

At that point, and to complete the panorama on the thermodynamic formalism for Lorenz map, we mention the existence of two others works on that topic (\cite{juliano-thesis,luciano-thesis}) we are aware on. As far as we know, both also deal with that question of global equilibrium state. 
Nevertheless, they are not yet  published nor exist as preprints. 
For that reason we shall not comment more on these works.

\bigskip
To compare relative and global equilibrium states, one strategy is obviously to find conditions  on the potential ensuring that any equilibrium state must be a good measure. 

Following \cite{buzzi-thermo}, a ``bad'' measure is exactly a measure shadowed by the critical set. Then Lemma 7.3 in \cite{buzzi-thermo} shows that any bad measure must have zero entropy\footnote{the proof is done with $f'(0)=0$ but the main ingredient is $|Df(0)|=+\8$}.

Obviously, a simple condition is to work with potentials such that 
$$\max A_{0}-\min A_{0}<h_{top}(F).$$
This ensures that any equilibrium state must have positive entropy, thus be a good measure. Nevertheless such a condition is very restrictive because it forbids to deal with $\be. A_{0}$, letting $\be\to+\8$.

\medskip
We present here another way to ensures that a relative equilibrium is also a global one.  
If $\mu$ is a bad measure and an equilibrium state for $A_{0}$, then it is also a $A_{0}$-maximizing measure:
$$\int A_{0}\,d\mu=\max\{\int A_{0}\,d\nu\}.$$
Therefore, a way to ensure that the relative equilibrium state among good measures is also a/the  global equilibrium state, is to find conditions on $A_{0}$ yielding that any maximizing measure is a good measure. 

We remind that it is conjectured for uniformly hyperbolic dynamical systems that any generic potential in Lipschitz (or H\"older) norm admits a unique maximizing measure and it has periodic support\footnote{\ie supported on a single periodic orbit}. Obviously any $F$-invariant measure with periodic support is a good measure. 
Therefore, the goal is to find reasonable conditions on $A_{0}$ such that it admits an unique maximizing measure, and it has periodic support. 

It is very easy to construct potentials $A_{0}$ with this property. But then, the question is to describe how big this set of potential is. 
This question cannot be solved so easily, but we can however give some elements of answer. 

We remind (see the proof of Corollary \ref{cor-entropy-scs}) that there is a canonical semi-conjugacy between a subshift $\K$ in $\{0,1\}^{\N}$ (not of finite type) and $([-1,1],f)$. Any point, except the pre-critical orbit   has a unique coding, any preimage of 0 has two. 
Moreover, expansivity yields that the lift of the potential $A_{0}$ in that subshift is again H\"older. 

We remind that H\"older or Lipschitz functions are the same in the shift (up to a change of the definition of the distance). The lifted potential, say $\wt A_{0}$ is only defined on the associated subshift but it can be extended (in a Lipschitz way) such that any maximizing measure for the extension is also a maximizing measure for $\wt A_{0}$. This can be done setting 
$$\wt A_{ext}(x):=\min\{A(y)-d(x,y),\ y\in\K\ d(x,y)=d(x,\K)\}.$$

Then, we use the result in \cite{bousch-walters} saying that generically for the Walters norm (weaker than the Lipschitz but stronger than the $\CC^{0}$ norms) a potential admits a unique maximizing measure and it has periodic support.

Nevertheless, this result does not immediately yields that the set of potentials for $([-1,1],f)$ which admit a unique maximizing and with periodic support measure  is generic in the sense of the Walters norm for $([-1,1],f)$. Indeed a perturbation in the subshift $\K$ may change the potential in such a way that it cannot be seen as the lifted of a potential in $[-1,1]$. However, this result indicates that it is highly probable that potentials with a single maximizing measure and with periodic support form a big set of potentials. For these potentials, there exists a unique global equilibrium state.

%
%
%
%
%
%
%

\bibliographystyle{plain}
\bibliography{mabiblio}

\noindent
\texttt{R. Leplaideur. Laboratoire de Math\'ematiques,
UMR 6205,\\
 Universit\'e de
Bretagne Occidentale,
 6 av. Victor Le Gorgeu, CS 93837, F-29238 BREST Cedex 3\\
  renaud.leplaideur@univ-brest.fr}

\bigskip  
\noindent
\texttt{V. Pinheiro. }  

\end{document}